\documentclass[10pt]{article}
\usepackage[utf8]{inputenc}
\usepackage[american,british]{babel}
\usepackage{babel}
\usepackage[T1]{fontenc}
\usepackage{amsmath,amsthm,amsfonts,amssymb,amscd,dsfont}
\usepackage{mathrsfs}
\usepackage{mathtools}
\usepackage{stmaryrd}
\usepackage{physics}
\usepackage{subcaption}
\usepackage{graphicx}

\usepackage{hyperref}
\hypersetup{
    colorlinks=true,
    linkcolor=blue,
    citecolor=red
    }    
\usepackage[capitalize]{cleveref}  
\Crefname{equation}{}{}

\usepackage{enumitem}
\usepackage[table]{xcolor}
    
\usepackage{fullpage}
\usepackage{lastpage}    
    
\numberwithin{equation}{section}

\newtheoremstyle{break}%
  {}{}%
  {\itshape}{}%
  {\bfseries}{}
  {\newline}{}

\theoremstyle{break}

\newtheorem{definition}{Definition}[section]
\newtheorem{theorem}[definition]{Theorem}
\newtheorem{corollary}[definition]{Corollary}
\newtheorem{lemma}[definition]{Lemma}
\newtheorem{proposition}[definition]{Proposition}

\newtheorem{maintheorem}{Theorem}

\crefname{maintheorem}{theorem}{theorems}

\theoremstyle{plain}
\newtheorem{remark}[definition]{Remark}

\newcommand{\BB}[1]{\mathbb{#1}}

\newcommand{\N}{\mathbb{N}}

\newcommand{\Z}{\mathbb{Z}}
\newcommand{\R}{\mathbb{R}}
\newcommand{\C}{\mathbb{C}}
\newcommand{\E}{\mathbb{E}}

\renewcommand{\P}{\mathbb{P}}
\newcommand{\Tcal}{\mathcal{T}}
\newcommand{\Fcal}{\mathcal{F}}

\newcommand{\Kcal}{\mathcal{K}}

\newcommand{\Dcal}{\mathcal{D}}

\newcommand{\Gcal}{\mathcal{G}}
\newcommand{\bigO}{\mathcal{O}}

\newcommand{\Scal}{\mathcal{S}}

\newcommand{\Ecal}{\mathcal{E}}

\newcommand{\eps}{\epsilon}

\newcommand{\df}{\operatorname{d}\!}
\newcommand{\dx}{\df x}
\newcommand{\dy}{\df y}

\newcommand{\du}{\df u}
\renewcommand{\dv}{\df v}
\newcommand{\ds}{\df s}
\newcommand{\dr}{\df r}
\newcommand{\dt}{\df t}

\newcommand{\Lap}{\Delta}

\renewcommand{\d}{\, \mathrm{d}}
\renewcommand{\Re}{\mathrm{Re} \,}

\DeclareMathOperator{\diam}{\mathrm{diam}}
\DeclareMathOperator{\dist}{\mathrm{dist}}

\DeclareMathOperator{\supp}{\mathrm{supp}}

\DeclareMathOperator{\Cov}{\mathrm{Cov}}
\DeclareMathOperator{\Var}{\mathrm{Var}}

\newcommand{\cbr}[1]{\left\lbrace #1 \right\rbrace}
\newcommand{\rbr}[1]{\left( #1 \right)}
\newcommand{\abr}[1]{\left[ #1 \right]}

\renewcommand{\abs}[1]{\left | #1 \right |} 
\newcommand{\innerp}[2]{\left\langle #1 , #2 \right\rangle}
\renewcommand{\norm}[1]{\left\lVert#1\right\rVert}

\newcommand{\1}{\mathbf{1}}
\newcommand{\half} {\frac{1} {2} }

\begin{document}
\title{Reconstruction of log-correlated fields from multiplicative chaos measures}
\author{Sami Vihko$^1$}

\date{$^1$University of Helsinki, Department of Mathematics and Statistics, P.O. Box 68, FIN-00014 University of Helsinki, Finland \\[2ex]
    \today
}

\maketitle

\section*{Abstract}
We consider log-correlated random fields $X$ and the associated multiplicative chaos measures $\mu_{\gamma,X}$. Our results reconstruct the underlying field $X$ from the multiplicative chaos measure $\nu_{\gamma,X}$. The new feature of our results is that we allow the dimension $d$ to be arbitrary and cover also the critical case $\gamma=\sqrt{2d}$. In the sub-critical regime $\gamma<\sqrt{2d}$, we allow the fields to be mildly non-Gaussian, that is, the field has the decomposition $X=G+H$ with a log-correlated Gaussian field $G$ and a Hölder-continuos (not necessarily Gaussian) field $H$.

\tableofcontents
\newpage

\section{Introduction and main results}
\label{sec:Introduction and main results}

In this paper, we will consider log-correlated random fields $X$, that is, random fields with a logarithmic singularity on the diagonal of their covariance kernel, and corresponding multiplicative chaos. Formally the multiplicative chaos is the exponential ``$e^{\gamma X}$'', where $\gamma$ is a real parameter. These objects are encountered in diverse mathematical fields such as mathematical physics, random geometry and probabilistic number theory, see for example \cite{DuSh11a,AsJoKu11a,Be17a,AlRhVa13a,DuRhSh14a,FyKe14a,SaWe20a} and references therein. 

In the development of the theory of multiplicative chaos, the field $X$ is usually taken to be a log-correlated Gaussian field, which we denote by $G$ here. Perhaps the most canonical example of a log-correlated Gaussian field is the planar Gaussian free field (GFF). This is a stochastic process with a covariance kernel given by the Green's function of the Laplace operator\footnote{More precisely one usually takes the Green's function of the operator $-\Lap$} on a domain in the plane $\R^2$ with appropriate boundary conditions. It is mathematically a rough object and is not pointwise well-defined as will be discussed below. However, formally we can think of it as a two-dimensional generalization of Brownian motion, which is a random path. Thus, approximately we view the GFF as a rough random surface, see \Cref{fig:Simulation of the GFF}.

\begin{figure}[h!]
\centering
\begin{subfigure}{.42\textwidth}
  \centering
  \includegraphics[width=.9\linewidth]{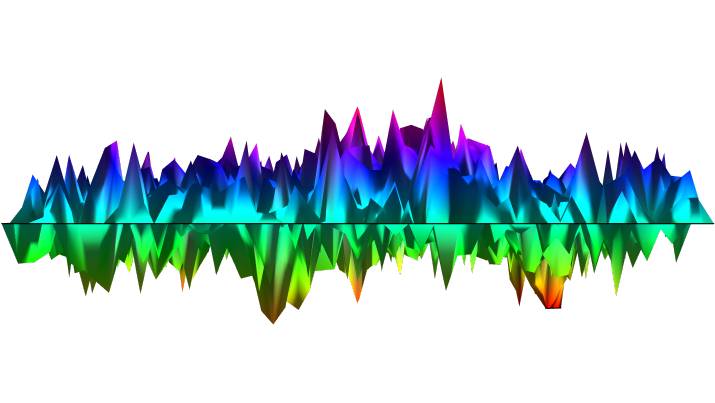}
  \caption{3D front view.}
  \label{subfig:GFF front}
\end{subfigure}
\begin{subfigure}{.42\textwidth}
  \centering
  \includegraphics[width=.9\linewidth]{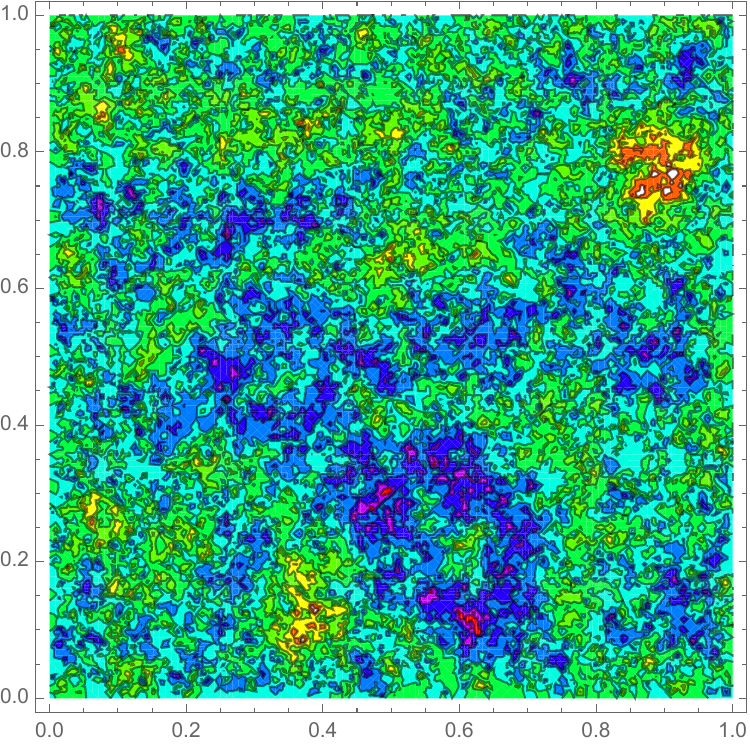}
  \caption{Contour plot}
  \label{subfig:GFF contour}
\end{subfigure}
\begin{subfigure}{.1\textwidth}
\centering
  \includegraphics[width=.5\linewidth]{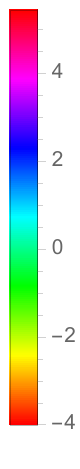}
\end{subfigure}
\caption{A simulation of the GFF on $[0,1]^2$. Two different graphic views of the same realization of an approximated GFF on $[0,1]^2$ with Dirichlet zero boundary conditions. The approximation used is the cut-off of the generalized Karhunen-Lo\`eve expansion at $(N,N)=(200,200)$ (see discussion after \Cref{eq:series expansion for log Gaussian fields}).}
\label{fig:Simulation of the GFF}
\end{figure}

The study of Gaussian multiplicative chaos (GMC) was started by Kahane in 1985 \cite{Ka85a} motivated by statistical models of turbulence. We refer to the review article \cite{RhVa14a} for more details and references about the early development of the theory. The theory is about fractal random measures, which are constructed from log-correlated Gaussian random fields $G$. Formally these are given by $e^{\gamma G(x)}dx$. There has been strong research interest in this area during the last two decades. Especially, the paper \cite{DuSh11a} raised interest in the subject again.  These fractal measures have some peculiar geometric properties and have also helped to study the properties of the underlying fields. We mention the property that in the sub-critical case the measures give full mass to the set of the thick points of the underlying field $G$. This result goes back to Kahane's original work\cite{Ka85a}. The $\gamma$-thick points $x$ in the domain of definition of the field satisfy
\begin{align*}
\lim_{\eps\to 0}\frac{G_\eps(x)}{\log(\eps^{-1})}=\gamma,
\end{align*}
where the $G_\eps$ is a suitable approximation of the field $G$ and $\gamma\in [0,\sqrt{2d})$. The variance of $G_\eps$ is of order $\log(\eps^{-1})$. Thus, formally the field $G$ takes exceptionally large values on the set of thick points.
 
Since the multiplicative chaos measure is constructed from a log-correlated field, it is natural to ask: What information about this underlying field does the measure carry? It turns out that it carries all the necessary information to reconstruct the whole field. Our main results give this reconstruction locally on a slightly smaller domain. Some earlier results already exist. The first reconstruction result was given in \cite{BeShSu23a}. Another reconstruction is given in \cite{ArJu21a}, where the authors consider imaginary multiplicative chaos. In this paper we are only interested in the real multiplicative chaos. The setup in \cite{BeShSu23a} is equivalent to ours, except that they consider only the dimension $d=2$ and the sub-critical regime $\gamma\in[0,\sqrt{2d})$. Although most concrete examples of log-correlated fields are in dimensions $d=1$ or $d=2$, many of the properties of log-correlated fields are independent of the dimension $d$. Thus, it would be desirable to have a reconstruction argument that is also independent of the dimension $d$.  This is one of the new features of the present paper. Another novelty is that our results apply also to the critical case $\gamma=\sqrt{2d}$. Our proofs rely on the scaling properties of the $\star$-scale invariant fields, which is a special class of log-correlated Gaussian fields, whereas the proofs in \cite{BeShSu23a} rely on the conformal invariance and Markov property of the $2d$-GFF. We use the coupling of the log-correlated field $X$ to a $\star$-scale invariant field $S$ to translate the argument of the proof to a wider class of fields. The log-correlated Gaussian fields $G$ are defined below and the $\star$-scale invariant fields $S$ are defined in \Cref{sec:star scale invariant fields}. The most general class of $\log$-correlated fields $X$ we consider are mildly non-Gaussian, and we define this concept in the next subsection before we state our main results. 

Given a bounded domain $D\in\R^d$ and a covariance kernel $C_G\colon D\times  D \to \R$ defined by
\begin{align}
\label{eq:definition of the covariance kernel of a log correlated Gaussian field}
C_G(x,y):=\log(|x-y|^{-1})+g_G(x,y),
\end{align} 
where $g$ is taken to be continuous at least on the interior of $D$, we call a random Schwartz distribution $G$ a log-correlated \emph{Gaussian} field if for all $f\in \Dcal(D)=C_c^\infty(D)$ the random variable $\innerp{G}{f}$ is Gaussian with variance 
\begin{align*}
\Var{(\innerp{G}{f})}=\int_{D\times D}f(x)f(y)C_G(x,y)\dx\dy:=\Kcal_G(f,f).
\end{align*}
Thus, the field is a random element in the space of ordinary distributions $\Dcal'(D)$. Above $\innerp{\cdot}{\cdot}$ is the canonical dual pairing. 
\noindent
We could and also will when necessary consider the fields $G$ in the whole $\R^d$ by setting the field to zero on ${\R^d\setminus D}$ and similarly for the covariance kernel. In this generality, probably the simplest way to prove the existence of such a field would be to invoke the Bochner-Minlos theorem with the functional $S(f)=\exp(-\frac{1}{2}\Kcal_G(f,f))$. This proves that there exists a probability measure on $\Dcal'(D)$, which would then be identified as the law of the desired field.

A more concrete way to construct the field $G$ is via a generalized Karhunen-Lo\`eve 
expansion
\begin{align}
\label{eq:series expansion for log Gaussian fields}
G(x):=\sum_{n=1}^\infty A_n \sqrt{\lambda_n}\varphi_n(x), \quad x\in D,
\end{align}
where $\{\lambda_n\}_{n\in\N}$ is a decreasing sequence of positive real numbers, $\{\varphi_n\}_{n\in\N}$ is a certain orthonormal set of vectors in $L^2(D)$ and $\{A_n\}$ is an i.i.d sequence of standard Gaussian random variables. In the case of log-correlated fields, the above series is formal and only converges in the sense of Schwartz distributions. The sequences $\{\lambda_n\}_{n\in\N}$ and $\{\varphi_n\}_{n\in\N}$ are the non-zero eigenvalues and corresponding eigenvectors of the Hilbert-Schmidt operator with integral kernel $C_G(x,y)$. This is a more convenient approach for proving that the field $G$ exists also in much nicer distribution spaces than $\Dcal'$, for example in some Sobolev spaces. Under the assumptions that $D$ is a simply connected domain and $g\in L^1(D\times D)\cap C(D\times D)$ is bounded from above and the field $G$ is continued as zero outside $D$, the above formal series converges in $H^{-\eps}(\R^d)$ for every $\eps>0$ and that the limit is a non-trivial $H^{-\eps}(\R^d)$-valued random variable \cite[Proposition 2.3]{JuSaWe20a}. Note that in this paper we do not need to assume simply connectedness since we do not reconstruct the field in the Sobolev space $H^{-\eps}(\R^d)$, but instead as an ordinary Schwartz distribution. For example, for the $2d$-GFF with zero Dirichlet boundary conditions on $[0,1]^2$, the functions $\varphi_n$ are the eigenfunctions of the Laplacian with the same boundary conditions $\varphi_{k,l}(x,y)=2\sin(k\pi x)\sin(l\pi y)$ with $k,l=1,2,\dots$ and the corresponding eigenvalues are the inverses $\lambda_{k,l}^{-1}$ of the eigenvalues of the Laplacian $\lambda_{k,l}=\half\pi^2(k^2+l^2)$.

We also need the decomposition of Gaussian fields \cite[Theorem A]{JuSaWe19a}. By this result, given a log-correlated Gaussian field $G$ we can find a $\star$-scale invariant field $S$ and a H\"older-continuous field $H$ defined on the same probability space such that
\begin{align*}
G=S+H.
\end{align*}
Note that the fields above are not necessarily independent. This result is valid for two more general centred Gaussian fields $G_1$ and $G_2$ than the log-correlated ones $G$ and $S$ here, and assumes that the difference of the covariances belongs to the space $H_{loc}^{d+\eps}(D\times D)$ for all $\eps>0$. Here the local Sobolev-space is defined so that a Schwartz distribution $\lambda$ belongs to $H_{loc}^s(D)$ if for every $\varphi\in C_c^\infty(D)$ we have $\varphi\lambda\in H^s(\R^d)$. In our case the difference of the covariances reduces to the difference $g_G-g_S$. However, for the $\star$-scale invariant part $S$ the function $g_S$ is usually more regular. Therefore, from now on in this paper we assume that $g_G\in H_{loc}^{d+\eps}(D\times D)$.

As mentioned already, the fields considered here are not point-wise well-defined random functions, but rather random Schwartz distributions. Thus, also their exponentials are ill-defined. The GMC is usually constructed as a measure by a limiting argument from the approximations of $G$. More precisely, in the sub-critical case $\gamma<\sqrt{2d}$ given a suitable approximation $G_\eps$ of $G$ and a Radon measure $\mu$ supported on $D$ we consider  the limit
\begin{align}
\nu_{\gamma,G,\mu}(\dx):=\lim_{\eps\to 0}\nu_{\gamma,\eps,G,\mu}:=\lim_{\eps\to 0} \frac{e^{\gamma G_\eps(x)}}{\E[e^{\gamma G_\eps(x)}]}\mu(\dx)=\lim_{\eps\to 0} e^{\gamma G_\eps(x)-\frac{\gamma^2}{2}\E[(G_\eps(x))^2]}\mu(\dx),
\end{align}
in the space of Radon measures supported in $D$ in suitable stochastic sense (see \Cref{subsec:Gaussian multiplicative chaos} for more details). With appropriate conditions, the limit defines a random Radon measure. We call the Radon measure $\mu$ a reference measure. In our results, we will take $\mu$ to be the Lebesgue measure and only shortly discuss after the proofs what parts of the argument need to be changed for a more general reference measure. We will also drop the subscript corresponding to the reference measure from the notation. In the approximations it is also convenient to drop the subscript for the field, when no confusion is possible. In the critical case $\gamma=\gamma_c:=\sqrt{2d}$, the approximations $\nu_{\gamma,\eps}$ need to be multiplied by a renormalization factor $C_\eps$ in order to obtain a non-trivial limit. In this paper, we use the deterministic renormalization $C_\eps=\sqrt{\log(\eps^{-1})}$, which is called the Seneta-Heyde renormalization. Thus, for us 
\begin{align*}
\nu_{\gamma_c,G}(\dx):=\lim_{\eps\to 0} \sqrt{\log(\eps^{-1})}\nu_{\gamma_c,\eps}(\dx).
\end{align*} 
The results about GMC measures we need will be reviewed in \Cref{subsec:Gaussian multiplicative chaos}. For more detailed expositions we refer the reader to \cite{Ka85a,RoVa10a,RhVa14a,Sh16a,Be17a} for the sub-critical case and to \cite{DuRhSh14b,DuRhSh14a,JuSa17a,JuSaWe19a,Po20a, La24a} for the critical case.

In many interesting models the field is not a log-correlated Gaussian field, but perturbation of such a field. In this paper, we consider the case that the perturbation is a Hölder-continuous field and call the field $X$ then a mildly non-Gaussian log-correlated field. We will give the necessary precise definitions in the next subsection before stating our main result covering also this case. 

\subsection{Main results}
\label{subsec:Main results}
In this section, we state our main results. We will take $D\subset\R^d$ to be a bounded domain. \Cref{thm:field from the measure general log correlated Gaussian case} describes the reconstruction of the underlying field from the corresponding GMC measure in the sub-critical and critical regime $\gamma\in[0,\sqrt{2d}]$ . Thus, the field is determined by the GMC measure. Furthermore, this reconstruction also gives that the field is measurable with respect to the multiplicative chaos, that is $G$ is measurable with respect to $\sigma(\mu_{\gamma,G})$. 
\begin{maintheorem}
\label{thm:field from the measure general log correlated Gaussian case}
Let $G$ be log-correlated Gaussian field with the decomposition $G=S+H$, where $S$ is $\star$-scale invariant field and $H$ is a.s. Hölder-continuous Gaussian field (as mentioned above, this means we assume that the function $g_G$ in \ref{eq: def function g} belongs to $H_{loc}^{d+\delta}(D\times D)$ for all $\delta>0$) and $\nu_{\gamma,G}$ be a GMC measure constructed from $G$. Let $\eta\in C_c(\R^d)$ with $\eta\geq 0$, $\int\eta dx=1$ and $\supp(\eta)\subset B_1(0)$, and $\gamma\in [0,\sqrt{2d}]$. Then there exists $\eps_{0}\in (0,1)$ such that for $0<\eps<\eps_0$ there exists a deterministic function $F_{\gamma,\eps,\eta}(x)$ such that for any $\psi \in C_c^\infty(D)$ with $\dist(\supp(\psi),\partial D)>\eps_0$ we have
\begin{align*}
\int_{D}\psi(x)\rbr{\frac{1}{\gamma}\log\abr{\int_{\R^d}\eta_{\eps}(y-x)\nu_{\gamma,G}(\dy )}-F_{\gamma,\eps,\eta}(x)}dx\overset{\eps\to 0}{\longrightarrow} \innerp{G}{\psi}
\end{align*}
in probability. Above $\eta_\eps(\cdot):=\eps^{-d}\eta(\cdot /\eps)$.

The counter term function $F_{\gamma,\eps,\eta}$ will be different for different fields $G$, but it will only depend on their distributions, not specific realizations. 
 
\end{maintheorem}
\noindent
The choice of the support of $\psi$ guarantees that there always exists small enough $\eps>0$ such that $\supp(\eta_\eps(\cdot-x))\subset D$ for all $x\in\supp(\psi)$ and thus the integral inside the logarithm in the above theorem is well-defined if we set $\nu_{\gamma,\cdot}=0$ in $\R^d\setminus D$. Note that as will be explained in \Cref{subsec:Gaussian multiplicative chaos} the approximations of critical GMC for general log-correlated Gaussian field $G$ converge in compact subsets of $D$ \cite{JuSaWe19a}. However, our assumptions guarantee that we are always within some compact set in the reconstruction.

\Cref{thm:field from the measure for mildly non-Gaussian case} is a similar result for the class of mildly non-Gaussian log-correlated random fields. In this case we only consider the sub-critical regime $\gamma<\sqrt{2d}$ for simplicity. However, for \Cref{def: non gaussian multiplicative chaos 1} $\gamma=\gamma_c$ makes no difference and in the example where we discuss using \ref{def: non gaussian multiplicative chaos 2} it has been shown that the limit in \Cref{def: non gaussian multiplicative chaos 2} also exists in the critical case.  We need to specify what we mean by mild non-Gaussianity.  
\begin{definition}[Mildly non-Gaussian log-correlated random field]
We say that a generalized random field (a random Schwartz distribution) $X$ is a mildly non-Gaussian log-correlated random field if there exist a log-correlated Gaussian field $G$ and an a.s Hölder-continuous field $H$ living on the same probability space as $X$ and $G$ such that
\begin{align*}
X=G+H.
\end{align*}
The processes above need not be independent. 
\end{definition}
\begin{remark}
In the mildly non-Gaussian case the term log-correlated random field is not completely analogous with the Gaussian case since although there will be a logarithmic part in the covariance, there is no guarantee that the rest of the covariance is as nice as the function $g$ in the Gaussian case. 
\end{remark}
\begin{remark}
As stated above, by the results in \cite{JuSaWe19a} there exists a decomposition $G=S+H'$ under some extra assumptions, where $S$ is a $\star$-scale invariant field, and $H'$ is an almost surely Hölder-continuous Gaussian process living on the same probability space as $S$. Thus, under the assumptions mentioned we could take $G$ to be $\star$-scale invariant field in the above definition.
\end{remark}
\noindent
Even though some results have been obtained \cite{Ju20a}, the canonical theory of non-Gaussian Multiplicative chaos does not exist, so we need to discuss how we will define the multiplicative chaos in the mildly non-Gaussian case. We will give two options here. Our main results will be stated according to the first option and discuss after the proofs, what further assumptions would be needed if we used the second option. The simplest option would be to define the following.
\begin{definition}[Multiplicative chaos for a mildly non-Gaussian field $X$, Option 1]
\label{def: non gaussian multiplicative chaos 1}
For a mildly non-Gaussian field $X=G+H$ we set
\begin{align*}
\nu_{\gamma,X}(\dx ):=e^{\gamma H(x)}\nu_{\gamma,G}(\dx ).
\end{align*} 
\end{definition}
\noindent
The downside of this definition is that now the expectation of this random measure is not the reference measure, that is, the Lebesgue measure. Neither is this a natural generalization of the GMC theory. The second option would fix the above, but would require further assumption in our results. 
\begin{definition}[Multiplicative chaos for a mildly non-Gaussian field $X$, Option 2]
\label{def: non gaussian multiplicative chaos 2}
For the mildly non-Gaussian field $X$ we set
\begin{align*}
\nu_{\gamma,X}(\dx ):=\lim_{\delta\to 0}\frac{e^{\gamma X_\delta(x)}}{\E[e^{\gamma X_\delta(x)}]}\dx 
\end{align*}
provided such a limit exists in some suitable topology in the space of Radon measures and almost surely (possibly along a subsequence) for some approximation $X_\delta$ of $X$.
\end{definition}
\begin{remark}
Since there is no suitable canonical theory of non-Gaussian multiplicative chaos, and we have not assumed anything more than the decomposition $X=S+H$, the existence of the limit in the above definition needs to be checked case by case. 
\end{remark}
\noindent
In \Cref{def: non gaussian multiplicative chaos 2}, the weak topology would be a natural choice since the majority of the theory in the Gaussian case has been done in this topology. For our needs in this paper, a little less than the weak topology is enough since we always integrate against a compactly supported test functions. Thus, we could use the vague topology of convergence of measures.
Now we can state the result for the mildly non-Gaussian case.
\begin{maintheorem}
\label{thm:field from the measure for mildly non-Gaussian case}
Let $X=G+H$ be a mildly non-Gaussian log-correlated field and the multiplicative chaos $\nu_{\gamma,X}$ as in \Cref{def: non gaussian multiplicative chaos 1}. Assume that $G$ satisfies the assumptions in \ref{thm:field from the measure general log correlated Gaussian case}, in particular $g_G\in H_{loc}^{d+\delta}(D\times D)$ for all $\delta>0$. Also let $\eta\in C_c(\R^d)$ with $\eta\geq 0$, $\int\eta dx=1$ and $\supp(\eta)\subset B_1(0)$, and $\gamma\in [0,\sqrt{2d})$. Then there exists $\eps_{0}\in (0,1)$ such that for $0<\eps<\eps_0$ there exists a deterministic function $F_{\gamma,\eps,\eta}(x)$ such that for any $\psi \in C_c^\infty(D)$ with $\dist(\supp(\psi),\partial D)>\eps_0$ we have
\begin{align*}
\int_D \psi(x)\rbr{\frac{1}{\gamma}\log\abr{\int_{\R^d}\eta_\eps(y-x)\nu_{\gamma,X}(\dy )}-F_{\gamma,\eps,\eta}(x)}dx\overset{\eps\to 0}{\longrightarrow} \innerp{X}{\psi}
\end{align*}
in probability. Above $\eta_\eps(x):=\eps^{-d}\eta(x/\eps)$.

The counter term function $F_{\gamma,\eps,\eta}$ will be different for different fields $X$, but it will only depend on their distributions not specific realizations. 
\end{maintheorem} 

The results in \cite{BeShSu23a} cover \Cref{thm:field from the measure general log correlated Gaussian case,thm:field from the measure for mildly non-Gaussian case} in the case $d=2$ and $\gamma<\sqrt{2d}$. Their proofs rely heavily on the coupling to the Gaussian free field its Markov property and conformal invariance. Parallel results also exist for $d\geq 2$ in the frame work of imaginary GMC, that is, taking $\gamma=i \beta$ for $\beta\in [0,\sqrt{d})$ \cite{ArJu21a}. In this case, the GMC is not a measure any more, but a complex-valued Schwartz distribution, see \cite{JuSaWe20a} for more details. In the case of imaginary GMC a direct local reconstruction is not possible. This is essentially, because of the periodicity of the real map $x\mapsto e^{ix}$. Instead, in \cite{ArJu21a} the authors first reconstruct the gradient of the underlying field, which then determine the field given some global conditions like boundary conditions on the domain.

\subsection*{Structure of the rest of the paper and some notation}
The rest of the paper is structured in the following way.
\Cref{sec:Preliminaries} will review some basic results about $\star$-scale invariant fields, continuous Gaussian processes and GMC measures and derive the tools needed to prove the main theorems. \Cref{sec:Proofs of the main results} is dedicated for the proofs of the main theorems and related discussions. We will first state a similar  result for the $\star$-scale invariant fields (\Cref{thm:field from the measure star scale invariant case}) and then prove \Cref{thm:field from the measure general log correlated Gaussian case,thm:field from the measure for mildly non-Gaussian case} assuming this result. After the proofs we will discuss the additional assumption needed if we use \Cref{def: non gaussian multiplicative chaos 2} in \Cref{thm:field from the measure for mildly non-Gaussian case}. At the end of \Cref{sec:Proofs of the main results} we will prove \Cref{thm:field from the measure star scale invariant case} and discuss the changes needed in the argument for a general reference measure.  \Cref{sec:Applications} will provide two examples of mildly non-Gaussian fields that fit into our framework. One of the examples, namely the sine-Gordon field comes from quantum/statistical field theory. The other example is the logarithm of the absolute value of the randomized Riemann zeta function. We will also give two applications of our methods. Both of these have the property that a small set of parameters will determine the object of interest for all other relevant values of the parameter. More precisely, we show with a concrete construction that one value say $\gamma_0$ determines the GMC measure for all other appropriate values of the parameter $\gamma$ and that the thick points determine in certain sense the whole log-correlated Gaussian field $G$.

In this paper (except in \Cref{sec:Examples}), the following notation will always be applied. $G$ denotes a Gaussian field, $S$ denotes a $\star$-scale invariant field and $H$ denotes a (not necessarily Gaussian) Hölder-continuous random field, that is, a process with a.s. surely Hölder-continuous realizations. $X$ will denote a generic random variable or a mildly non-Gaussian log-correlated field and it will be clear from context, which it is. In  \Cref{sec:Examples}, which gives the examples of mildly non-Gaussian fields we will use notation that is more in line with the papers that prove the coupling results of the given fields. $C$ (with possible subscripts to highlight dependence on something) will denote a positive constant and we will allow it to change from (in)equality to (in)equality in the estimates.  

\subsection*{Acknowledments}

The author was supported by the Emil Aaltonen Foundation,  Academy of Finland through the Grant 348452 and the Academy of Finland CoE FiRST. The author is also thankful for Christian Webb and Eero Saksman for helpful discussions and comments on the preliminary versions of this paper. 

\section{Preliminaries}
\label{sec:Preliminaries}
In this section, we will review and develop the tools that we need in the proofs of the main theorems. 
\subsection{$\star$-scale invariant fields}
\label{sec:star scale invariant fields}
The class of $\star$-scale invariant fields $S$ is a special class of log-correlated Gaussian fields. First, let $k\colon \R^d \to\R$ be a $\alpha$-H\"older-continuous with $0<\alpha\leq 1$, rotationally symmetric function with compact support in $B_1(0)$ and $k(0)=1$, and for which $(x,y)\to k(x-y)$ is a covariance on $\R^d\times \R^d$. We call such a function $k$ a seed covariance function following the terminology in \cite{JuSaWe19a}. To see that such a function $k$ exists, we can first take  $k:=u*u$ with rotationally symmetric, H\"older-continuous function $u$ with compact support and satisfying $u(x)\geq 0$ for all $x\in\R^d$. Checking the covariance property and H\"older-continuity is straight forward and the rest of the properties we demand can be achieved by scaling and normalization. Then we have the following definition
\begin{definition}[$\star$-scale invariant Gaussian field]
Given a function $k$ as described above we define a $\star$-scale invariant Gaussian field $S$ to be a generalized Gaussian field with the covariance kernel
\begin{align}
\label{eq:Covariance integral for star scale invariant fields}
C_S(x,y):=\int_1^\infty k(t[x-y])\frac{\dt }{t}=\int_1^\infty k_1(t|x-y|)\frac{\dt }{t},
\end{align}
where $k_1\colon \R\to\R$ with $k(x):=k_1(|x|)$.
\end{definition}
\begin{remark}
Note that different parametrizations exist, and different regularity assumptions on $k$ may be used, see for example \cite{DuFa20a,DuRhSh14b,DuRhSh14a,JuSaWe19a}. These $\star$-scale invariant fields also fall into the original framework of Gaussian multiplicative chaos (GMC) created by Kahane \cite{Ka85a}. In other words, the covariance kernel $C_S$ is   of sigma-positive type. 
The notion of $\star$-scale invariance is originally related to random measures that satisfy a certain $\star$-equation, which is a continuous analogue of a similar discrete equation satisfied by the so called multiplicative cascades developed by Mandlebrot in \cite{Ma74a,Ma74b}. Roughly speaking the measures satisfying the continuous $\star$-equation are the GMC measures that can be constructed from the above $\star$-scale invariant fields up to an independent multiplicative random variable, see \cite{AlRhVa13a,RhSoVa14a} and references therein.
\end{remark}

$C_S$ is indeed a covariance, because $k_1(|\cdot-\cdot|)$ is, and sums and limits of covariances are covariances. We could argue directly by for example Minlos' theorem, that such a field exists, but we will soon see that $S$ is a log-correlated Gaussian field. Indeed, the covariance integral in \Cref{eq:Covariance integral for star scale invariant fields} can be rewritten as
\begin{align*}
\int_1^\infty k(s[x-y])\frac{\ds}{s}=&
\log(|x-y|^{-1})+\int_{|x-y|}^1(k_1(t)-1)\frac{\dt }{t}
\end{align*}
since $\supp(k_1)\subset [0,1)$. Thus, we have
\begin{align*}
C_S(x,y)=\log(|x-y|^{-1})+g_S(x,y),
\end{align*}
where 
\begin{align*}
g_S(x,y):=\int_{|x-y|}^1(k_1(t)-1)\frac{\dt }{t}.
\end{align*}
$g_S$ is continuous as can be easily checked. Thus, $S$ is a log-correlated Gaussian field.

\subsubsection{The cut-off approximation of the $\star$-scale invariant fields}
\label{sec:The cut off approximation of the star scale invariant fields}
As with the $\star$-scale invariant fields themselves, we will introduce the approximations in the whole space $\R^d$ and use appropriate restrictions when necessary without changing the notations. Consider the centred Gaussian process $(\eps,x)\mapsto S_\eps(x)$ on $(0,1)\times \R^d$ with covariance
\begin{align*}
\Cov(S_\eps(x),S_\delta(y))&=\E[S_\eps(x)S_\delta(y)]
\\
&:=\int_1^{\eps^{-1}\wedge \delta^{-1}}k_1(s|x-y|)\frac{\ds}{s}
\\
&:=K_{(\eps,\delta)}(x,y),
\end{align*}
where $s \wedge t:=\min(s,t)$. Note, that the above is again a covariance since we assumed $k_1(|\cdot-\cdot|)$ to be a covariance function. In the case, $\eps=\delta$ we omit one of the subscripts from above and write $K_{(\eps,\eps)}=K_\eps$. Such a process exists and has a.s. 
H\"older-continuous realizations in $x \in \R^d$ for every fixed $\epsilon>0$.
This follows from standard regularity theory for Gaussian processes or simple Kolmogorov-Chentsov type arguments \cite[Proposition B.2]{LaRhVa15a}. For every fixed $x\in \R^d$ with the ``time-change'' $\eps(t)=e^{-t}$, $(S_\eps(x))_{\eps\in(0,1)}\mapsto (\tilde{S}_t(x))_{t\in(0,\infty)}$ the process $\tilde{S}_t$ has the law of a standard Brownian motion. Thus, for fixed $x\in\R^d$ this is a real valued martingale. For fixed $\eps\in(0,1)$ we call $(S_\eps(x))_{x\in \R^d}$ the cut-off approximation of the $\star$-scale invariant field $S$. $(S_\eps(x))_{\eps\in(0,1)}$ is a martingale in the decreasing parameter $\eps$, that is, it is an ordinary martingale in the parameter $\eps^{-1}$ and it is straightforward to show that so is $\innerp{S_\eps}{f}$ for a fixed test function $f$. Thus, by simple $L^2$-boundedness argument one can readily show that $\innerp{S_\eps}{f}$ converges almost surely to random variable, that has the same distribution as $\innerp{S}{f}$.   

\subsubsection{Scaling and translation properties of $\star$-scale invariant fields}
\label{sec:Scaling properties of star-scale invariant fields}
To study the scaling properties of the $\star$-scale invariant fields and the corresponding GMC measures we define two auxiliary fields constructed from the field $S_\eps$. For $x,u\in \R^d$ and $0<\delta<\eps<1$ we define the processes
\begin{align}
\label{eq: Def Z}
(\delta,\eps,x) \mapsto & S_\delta(x+\eps u)-S_\eps(x+\eps u)=:Z_{(\eps,\delta,x)}(u)
\\
\label{eq:Def Y}
(\eps,x) \mapsto & S_\eps(x+\eps u)-S_{\eps}(x)=:Y_{\eps,x}(u).
\end{align}
We need the following basic facts.
\begin{proposition}
\label{prop: properties of the process Z}
For fixed parameters $x\in \R^d$ and $\eps\in (0,1)$, and the variables $u\in \R^d$ and $\delta\in (0,\eps)$
\begin{enumerate}
\item the processes
\begin{align*}
\quad (\delta,u)\mapsto Z_{(\eps,\delta,x)}(u) \quad &\text{and} \quad u\mapsto S_\eps(u) 
\end{align*}
are independent and similarly the processes
\begin{align*}
\quad (\delta,u)\mapsto Z_{(\eps,\delta,x)}(u) \quad &\text{and} \quad u\mapsto Y_{\eps,x}(u) 
\end{align*}
are also independent.
\item
Furthermore, for the processes 
\begin{align*}
\quad (\delta,u)\mapsto Z_{(\eps,\delta,x)}(u) \quad &\text{and} \quad (\delta,u)\mapsto S_{\frac{\delta}{\eps}}(u) 
\end{align*}
we have
\begin{align*}
(Z_{(\eps,\delta,x)}(u))_{\{0<\delta<\eps,u\in\R^d\}}\overset{d}{=}(S_{\frac{\delta}{\eps}}(u))_{\{0<\delta<\eps,u\in\R^d\}}.
\end{align*}
\end{enumerate}
\end{proposition}

\begin{proof}
All processes are Gaussian. Thus, we simply check the covariances. For the first claim we have
\begin{align*}
\Cov(Z_{(\eps,\delta,x)}(u),S_\eps(v))=K_{(\delta,\eps)}(x+\eps u,v)-K_\eps(x+\eps u,v)=0,
\end{align*}
because $\delta<\eps\Rightarrow \eps^{-1}\wedge\delta^{-1}=\eps^{-1}$. Similarly we have
\begin{align*}
\E[Z_{\delta,\eps,x}(u)Y_{\eps,x}(v)]=&K_{(\delta,\eps)}(x+\eps u,x+\eps v)-K_{\eps}(x+\eps u,x+\eps v)
\\
&-K_{(\delta,\eps)}(x+\eps u,x)+K_{\eps}(x+\eps u,x)=0.
\end{align*}
For the second claim we obtain
\begin{align*}
\Cov(Z_{(\eps,\delta,x)}(u),Z_{(\eps,\eta,x)}(v))=&K_{(\delta,\eta)}(\eps u,\eps v)-K_{\eps}(\eps u,\eps v)
\\
=&\int_{\eps^{-1}}^{\delta^{-1}\wedge\eta^{-1}}k_1(s\eps| u- v|)\frac{\ds}{s}
\\
=&\int_1^{\frac{\eps}{\delta}\wedge\frac{\eps}{\eta}}k_1(s|u-v|)\frac{\ds}{s}
\\
=&\Cov(S_{\frac{\delta}{\eps}}(u),S_{\frac{\eta}{\eps}}(v)),
\end{align*}
where we have used the translation invariance of $K$ and the fact that $\delta,\eta<\eps$.  
\end{proof}
Note that $Y_{\eps,x}$ and $S_\eps$ need not be independent. Indeed,
\begin{align*}
\E[Y_{\eps,x}(u)S_\eps(v)]=&K_{\eps}(x+\eps u,v)-K_\eps(x,v)
\\
=&\int_1^{\eps^{-1}}[k_1(s|x-v+\eps u|)-k_1(s|x-v|)]\frac{ds}{s}
\end{align*}
is not equal to zero identically for all $x,u,v\in \R^d$ and $\eps\in (0,1)$. 

\subsection{Convergence and extrema of continuous Gaussian processes}
\label{subsec:Convergence and extrema of continuous Gaussian processes}
In this section, we review some facts about convergence and extrema of Gaussian random variables with values on the space of continuous functions, that is, continuous Gaussian process. First we consider Gaussian random variables with values on the space of continuous functions on a compact metric space with the topology of uniform convergence. For the first topic of convergence in law of Gaussian processes, we need some notions from the general theory of convergence in law in metric spaces. In the applications, the metric space will be continuous functions on a subset of Euclidean space with a metric induced by the $\sup$-norm. To begin, we have the following definitions
\begin{definition}[Relative compactness in distribution]
A sequence $\{X_n\}_{n\in\N}$ of random elements in a metric space T is relatively compact if every subsequence has a further subsequence that converges in law (distribution).
\end{definition}
\begin{definition}[Tightness of a sequence of elements in a metric space]
We say that a sequence of random elements $\{X_n\}_{n\in\N}$ in a metric space $T$ is tight if
\begin{align*}
\sup_{K\subset T}\liminf_{n\to \infty}\P(X_n\in K)=1,
\end{align*}
where the supremum is  taken over all compact sets $K\subset T$. 
\end{definition}
Then we need two theorems 
\begin{theorem}[{{\cite[Theorem 16.2]{Ka02a}}}]
Let $\{X_n\}_{n\in \N}$ be a sequence of random elements in $C(K,T)$ (continuous functions on a compact metric space $K$ with values in an arbitrary metric space $T$). Then $X_n\overset{d}{\rightarrow}X$ iff $X_n\overset{fd}{\rightarrow}X$ and $\{X_n\}_{n \in \N}$ is relatively compact in distribution.
\end{theorem}
\begin{remark}
Above $X_n\overset{fd}{\rightarrow}X$ means the convergence of all finite dimensional distributions, that is, random vectors formed from an arbitrary number of point evaluations. For Gaussian random elements it is enough to check that the covariance converges point-wise.
\end{remark}
\begin{theorem}[{{\cite[Theorem 16.3]{Ka02a}}}]
For any sequence of random elements in a metric space $T$, tightness implies relative compactness in distribution and the two are equivalent if $T$ is Polish space. 
\end{theorem}
\noindent
Furthermore, there exists a Kolmogorov-Chentsov type criterion for the tightness. For simplicity, let $(X_n)_{n\in\N}$ be random elements in $C(K,\R)\equiv C(K)$ ($K\subset \R^d$ compact) (for the general case, see \cite[Theorem 16.9]{Ka02a}). Then we have
\begin{theorem}
\label{thm:Kolmogorov-chentsov type criteria for the tightness of sequences processes}
The sequence $\{X_n\}_{n\in\N}$ is tight if $X_n(0)$ is tight and 
\begin{align*}
\E[|X_n(x)-X_n(y)|^a]\leq C |x-y|^{d+b}
\end{align*}
for some suitable constants $C,a,b>0$ and all $x,y\in K$. 
\end{theorem}     
\noindent
Recall, that for a real centred Gaussian random variable $G$ with variance $\sigma^2$, we have the following relation for positive integer absolute moments $\E[|G|^p]=B_p \sigma^p$ with an explicit constant $B_p$ whose precise form is not important here. Therefore, if $\E[G^2]\leq A$ also $\E[|G|^p]\leq B_p A^{p/2}$. If $\alpha>0$ we can always find large enough $p$ such that $\alpha p/2>d+b$. Thus, we have the following simple corollary for Gaussian processes $\{G_n\}_{n\in N}$.

\begin{corollary}
\label{thm:Kolmogorov-chentsov type criteria for the tightness of Gaussian sequences}
The sequence $\{G_n\}_{n\in\N}$ is tight if $G_n(0)$ is tight and 
\begin{align*}
\E[|G_n(x)-G_n(y)|^2]\leq C |x-y|^{\alpha}
\end{align*}
for some suitable constants $C,\alpha>0$ and all $x,y\in K$. 
\end{corollary}     

As an application, we have the following proposition
\begin{proposition}
\label{prop: convergence of Yx and Yx'}
Fix $x\neq x' \in \R^d$. Then as $\eps\to 0$ the restrictions of the processes (defined in \Cref{eq:Def Y})
\begin{align*}
u\mapsto Y_{\eps,x}(u) \quad \text{and} \quad u\mapsto Y_{\eps,x'}(u)
\end{align*}
to $\overline{B_1(0)}$ converge jointly in distribution (in $C(\overline{B(0,1)})$) to two independent copies of centred Gaussian processes with covariance function
\begin{align*}
C(u,v):=\int_0^1\frac{k_1(s|u-v|)-k_1(s|u|)-k_1(s|v|)+1}{s}\ds.
\end{align*}
\end{proposition}

\begin{proof}
It is enough to consider an arbitrary deterministic sequence such that $\eps_n\rightarrow 0$ as $n\to\infty$ and the processes $G_n(u):=(Y_{\eps_n,x}(u),Y_{\eps_n,x'}(u))$. Then $G_n(0)=0$ for all $n$ and so $\{G_n(0)\}_{n\in\N}$ is in particular tight. By denoting for simplicity $G_n=(G_n^1,G_n^2)$ we have
\begin{align*}
\E\abr{G_n^1(u)G_n^1(v)}&=\E\abr{G_n^2(u)G_n^2(v)}
\\
&=\int_1^{\eps_n^{-1}}\frac{k_1(s\eps_n|u-v|)-k_1(s\eps_n|u|)-k_1(s\eps_n|v|)+1}{s}\ds
\\
&=\int_{\eps_n}^1\frac{k_1(s|u-v|)-k_1(s|u|)-k_1(s|v|)+1}{s}\ds,
\end{align*}
where we have used the basic properties of the covariance of $S_\eps$ and the fact that $k_1(0)=1$. From the above, we also readily see that $\E[(G_n^1(u))^2]=2\int_{\eps_n}^1[(1-k_1(s|u|))/s]\ds$ since $k_1(0)=1$. Thus, we have
\begin{align*}
\E\abr{|G_n(u)-G_n(v)|^2}&=\sum_{i=1}^2\E\abr{(G_n^i(u)-G_n^i(v))^2}
\\
&=\sum_{i=1}^2\rbr{\sum_{t\in\{u,v\}}\E\abr{G_n^i(t)^2}-2\E\abr{G_n^i(u)G_n^i(v)}}
\\
&=2\int_{\eps_n}^1\frac{1-k(s[u-v])}{s}\ds
\\
&\leq 2|u-v|^\alpha\int_{0}^1s^{\alpha-1}\ds
\\
&\leq C_{\alpha}|u-v|^{\alpha},
\end{align*}
where $\alpha$ is the H\"older-exponent of $k$.

The off-diagonal elements then become for example
\begin{align*}
\E[G_n^1(u)&G_n^2(v)]
\\
&=\int_1^{\eps_n^{-1}}[k_1(s|x-x'+\eps_n(u-v)|)-k_1(s|x-x' +\eps_n u|)
\\
&\qquad -k_1(s|x-x'-\eps_n v|)+k_1(s|x-x'|)]\frac{\ds}{s}
\\
&=\int_{\eps_n}^1[k_1(s|(x-x')\eps_n^{-1}+(u-v)|)-k_1(s|\eps_n^{-1}(x-x') + u|)
\\
&\qquad -k_1(s|\eps_n^{-1}(x-x')- v|)+k_1(s|\eps_n^{-1}(x-x')|)]\frac{\ds}{s}.
\end{align*}
For the integrand, we have 
\begin{align*}
\frac{1}{s}|k_1&(s|(x-x')\eps_n^{-1}+(u-v)|)-k_1(s|\eps_n^{-1}(x-x') + u|)
\\
&\qquad-k_1(s|\eps_n^{-1}(x-x')- v|)+k_1(s|\eps_n^{-1}(x-x')|)|
\\
&=\frac{1}{s}\big[|k(s[(x-x')\eps_n^{-1}+(u-v)])-k(s[\eps_n^{-1}(x-x') + u])|
\\
&\qquad+|k(s[\eps_n^{-1}(x-x')- v])-k(s[\eps_n^{-1}(x-x')])|\big]
\\
&\leq 2C_{\alpha}|v|^{\alpha}s^{\alpha-1}
\end{align*}
which is integrable on $(0,1)$. Thus, dominated convergence yields that the off-diagonal elements of the covariance go to zero as $n\to \infty$ since $\supp(k_1)\subset[0,1)$.

Therefore, as the covariance of $G_n$ converges to $C\otimes C$ (diagonal matrix with $C$ on the diagonal). The Kolmogorov-Chentsov type estimate above guarantees that there exists a H\"older-continuous Gaussian process with covariance $C$. Thus,  \Cref{thm:Kolmogorov-chentsov type criteria for the tightness of Gaussian sequences} and the theorems above it give the desired convergence in law of $G_n$ to $(G^1,G^2)$, where $G^1$ and $G^2$ are independent and identically distributed with covariance $C$.
\end{proof}

We also need to control the extrema of the process $\{Y_{\eps,x}(u)\}_{u\in\R^d}$ uniformly in $\eps$. For this, we need some classic results. Namely, Dudley's theorem (see  \cite[Section 1.3]{AdTa07a}) and the Borel-TIS inequality (see \cite[Section 2.1]{AdTa07a}). These are

\begin{theorem}[Dudley's theorem]
\label{thm:Dudley}
\vskip 0.1in
\noindent
Let $T$ be a metric space (such that $(T,d_G)$ (see below) is compact) and $G$ a centred real valued Gaussian process on $T$. Define another (pseudo) metric on T by $d_G(t,s):=\sqrt{\E[\{X(t)-X(s)\}^2]}$. Then there exists a universal constant $C>0$ such that
\begin{align*}
\E\abr{\sup_{t\in T}G(t)}\leq C\int_0^{\diam(T)/2}\sqrt{H_G(\eta)}\d\eta,
\end{align*}
where $H_G$ is the log-entropy corresponding to the metric $d_G$. It is defined by $H_G(\eta):=\log(N_G(\eta))$, where $N_G(\eta)$ is the minimal number of (closed) $d_G$-balls with radius $\eta$ needed to cover $T$
\end{theorem}
and 
\begin{theorem}[Borel-TIS inequality]
\label{thm:Borel TIS}
\vskip 0.1in
\noindent 
Let $T$ be as above and $G$ an a.s bounded centred Gaussian process on $T$. Then $\E[\sup_{t\in T}G(t)]<\infty$ and 
\begin{align*}
\P(\{\sup_{t\in T} G(t)-\E[\sup_{t\in T}G(t)]>u\})\leq e^{-\frac{u^2}{2\sigma_T^2}},
\end{align*}
where $\sigma_{T}^2:=\sup_{t\in T}\E[(G(t))^2]$.
\end{theorem}
\begin{remark}
In the case, $(T,d_G)$ is not assumed to be compact, \ref{thm:Dudley} is still true, but the upper limit of the integral needs to be $\infty$. The convergence of the entropy integral (sometimes called 
the Dudley integral with various upper bounds) implies also that $G$ has a version with almost all sample paths bounded and continuous \cite[Theorem 11.17]{LeTa91a}. As mentioned earlier, this could be used to check the continuity of the sample paths of $S_\eps$.
\end{remark}
\begin{remark}
The following useful properties are also stated in \cite[Section 2.1]{AdTa07a}. 
An easy corollary to Borel-TIS is 
\begin{align*}
\P\rbr{\sup_{t \in T} G(t)>u}\leq e^{-\frac{\rbr{u-\E\abr{\sup_{t\in T}G(t)}}^2}{2\sigma_T^2}}.
\end{align*}
We also have by symmetry
\begin{align*}
\P\rbr{\sup_{t\in T}|G(t)|>u}\leq 2\P\rbr{\sup_{t\in T}G(t)>u}.
\end{align*}
\end{remark}
We need these for the next result.
\begin{lemma}
\label{lemma: even moments of Y}
Let $Y_{\eps,x}(u)$ be the process defined in \Cref{eq:Def Y}. Then for each $k\in\Z_+$ and $a\in\R$ there exist constants $D_1,D_2,C_k>0$ independent of $\eps>0$ and $x\in\R^d$ such that 
\begin{align}
\label{eq: moments of the maximum of Y}
\E\abr{\sup_{|u|\leq 1}|Y_{\eps,x}(u)|^{2k}}\leq C_k,
\end{align}
and in addition
\begin{align}
\label{eq:exponential moments of the maximum of Y}
0<D_1\leq\E\abr{e^{a\sup_{|u|\leq 1}|Y_{\eps,x}(u)|}}\leq D_2<\infty.
\end{align}
\end{lemma}
\noindent
From the lemma above we obtain the direct corollary.
\begin{corollary}
The maximum over $B_1(0)$ of the process $|Y_{\eps,x}|$ has all positive moments. 
\end{corollary}
\begin{proof}[Proof of \Cref{lemma: even moments of Y}]
\noindent
We first note that 
\begin{align}
\label{eq: inequality for sup Y}
\sup_{|u|\leq 1}|Y_{\eps,x}(u)| \leq |\sup_{|u|\leq 1} Y_{\eps,x}(u)| + |\sup_{|u|\leq 1} -Y_{\eps,x}(u)|,
\end{align}
and $Y_{\eps,x}$ is centred so that $|\sup_{|u|\leq 1}\pm Y_{\eps,x}(u)|$ are identically distributed. Then we obtain (by $|a+b|^{2k}\leq2^{2k}(|a|+|b|)^{2k}$)
\begin{align*}
\E\abr{\sup_{|u|\leq 1}|Y_{\eps,x}(u)|^{2k}}\leq C_k\E\abr{\bigg(\sup_{|u|\leq 1}Y_{\eps,x}(u)\bigg)^{2k}}
\end{align*}
for some $C_k>0$. 

Noting that the proof of \Cref{prop: convergence of Yx and Yx'} implies that 
\[d_{Y_{\eps,x}}(u,v)\leq C_\alpha|u-v|^{\alpha/2},\]
where $C_\alpha$ can be chosen independent of $\eps>0$ and $x,u,v\in\R^d$. This implies that
\begin{align*}
B_{Y_{\eps,x}}(0,r)=\{D_{Y_{\eps,x}}(0,u)<r\}\supset\{C_{\alpha}|u|^{\frac{\alpha}{2}}<r\}=B_{[r/C_\alpha]^{\frac{2}{\alpha}}}(0)
\end{align*}
where $B_R(0)$ is the Euclidean ball of radius $R>0$. In particular, $N_{Y_{\eps,x}}(\eta)<N([\eta/ C_\alpha]^{2/\alpha})$, where $N_{Y_{\eps,x}}(\cdot)$ and $N(\cdot)$ are the number of $d_{Y_{\eps,x}}$- and Euclidean balls respectively needed to cover the unit ball of $\R^d$. Noting that the number of Euclidean balls of radius $r$ needed to cover the unit ball is comparable to $1/r^d$ and using  Dudley's theorem we obtain
\begin{align*}
\E\abr{\sup_{|u|\leq 1}Y_{\eps,x}(u)}&\leq C\int_0^\infty\sqrt{\log(N(\eta))}\d\eta
\\
&=C\rbr{1+\int_0^1\sqrt{\log(\eta^{-1})}\d\eta}
\\
&<\infty,
\end{align*}
where the constant $C>0$ is independent of $\eps$ and $x$. 
Then noting that $\sigma^2:=\sup_{|u|\leq 1}\E[(Y_{\eps,x}(u))^2]<\infty$ and $E_Y:=\E[\sup_{|u|\leq 1}Y_{\eps,x}(u)]<\infty$ uniformly in $\eps>0$ and $x\in\R^d$ by previous estimates. Thus, for the first claim we obtain
\begin{align*}  
\E\abr{\sup_{|u|\leq 1}|Y_{\eps,x}(u)|^{2k}}&\leq C_k\E\abr{\bigg(\sup_{|u|\leq 1}Y_{\eps,x}(u)\bigg)^{2k}}
\\
&\leq C_k\int_0^\infty\P\rbr{\bigg(\sup_{|u|\leq 1}Y_{\eps,x}(u)\bigg)^{2k}>r}\dr
\\
&\leq C_k\int_0^\infty e^{-\frac{(v-E_Y)^2}{2\sigma}}v^{2k-1}\dv
\\
&<\infty
\end{align*}
where the constant $C_k $ is independent of $\eps$ and $x$. Note that since the supremum is positive we have simply solved the inequality inside the probability for the supremum and then changed the integration variable. The inequality leading to the last integral uses the corollary of the Borel-TIS inequality stated in one of the remarks following it.

Then consider the second claim. For $a<0$ we only need to prove the lower bound and for $a\geq 0$ we only need to prove the upper bound. First let $a<0$ and note that 
\[\E\abr{\sup_{|u|\leq 1}|Y_{\eps,x}(u)|}\leq 2\E\abr{|\sup_{|u|\leq 1}Y_{\eps,x}(u)|}=2\E\abr{\sup_{|u|\leq 1} Y_{\eps,x}(u)},\] 
because the supremum is positive since $Y_{\eps,x}(0)=0$. 
Then
\begin{align*}
\E\abr{e^{a\sup_{|u|\leq 1}|Y_{\eps,x}(u)|}}\geq e^{a\E[\sup_{|u|\leq 1}|Y_{\eps,x}(u)|]}\geq e^{2a\E[\sup_{|u|\leq 1} Y_{\eps,x}(u)]}>0
\end{align*}
by the above argument and previous estimates. The first inequality above is the Jensen inequality. 

For $a>0$ we have using similar manipulation as for the first claim
\begin{align*}
\E\abr{e^{a\sup_{|u|\leq 1}|Y_{\eps,x}(u)|}}&\leq\E\abr{e^{2a\sup_{|u|\leq 1}Y_{\eps,x}(u)}}
\\
&\leq  \int_0^\infty \P\rbr{e^{2a \sup_{|u|\leq 1}Y_{\eps,x}(u)}>r}\dr
\\
&\leq 2a\int_{\R}e^{-\frac{(v-E_Y)^2}{\sigma^2}}e^{2av}\dv
\\
&<\infty
\end{align*}
uniformly in $\eps$ and $x$. Above we used \Cref{eq: inequality for sup Y} and the properties mentioned below it and Cauchy-Schwarz.
\end{proof}

\subsection{Gaussian multiplicative chaos (GMC)}
\label{subsec:Gaussian multiplicative chaos}
As mentioned in the introduction, given a log-correlated Gaussian field $G$ we would like to define the multiplicative chaos as the exponential $e^{ \gamma G}$ for appropriate real parameters $\gamma$. However, since $G$ is a Schwartz distribution\footnote{Here by Schwartz distribution we mean an element of $\Dcal'$. These are often also called ordinary distributions and a Schwartz distribution might be used for an element of $\Scal'$, which we would call a tempered distribution.} the exponential is ill-defined. The standard definition is to consider the limit of the exponential of a suitable regularization or approximation $G_\eps$ of the log-correlated field suitably normalized and integrated against a reference measure. We will restrict the discussion to the case, where the reference measure is the Lebesgue measure and only briefly mention the definitions in the general case further below. In \Cref{sec:Generalization to more general reference measures than the Lebesgue measure} we will comment what parts of the arguments of the proofs need to be modified for more general reference measures.

We will be working with the so called sub-critical and critical GMC, that is, $\gamma\leq \sqrt{2d}$. In the sub-critical regime $\gamma<\sqrt{2d}$, to construct the GMC measures, we consider the limit
\begin{align*}
\lim_{\eps \to 0} \nu_{\gamma,\eps}(\dx):=\lim_{\eps\to 0}e^{\gamma G_\eps(x)-\frac{\gamma^2}{2}\E[(G_\eps(x))^2]}\dx.
\end{align*}
for an appropriate approximation $G_\eps$ of the log-correlated field $G$. Given a bounded domain  $D\subset\R^d$, this limit exists in probability in the weak-topology of convergence of measures with support in the domain $D$ for the convolution approximations \cite[Theorem 1.1]{Be17a}. Furthermore, the limit defines a non-trivial random measure. We will discuss more general results after we have stated the convergence results for $\star$-scale invariant fields that we need in this paper. 

In the critical case $\gamma=\gamma_c:=\sqrt{2d}$, the approximations of the GMC measure need to be renormalized to yield non-trivial result. There are two different renormalization schemes. The deterministic renormalization that we will use in this paper is the so called Seneta-Heyde renormalization. Up to deterministic multiplicative constant the other renormalization converges to the same measure. In the deterministic case we consider the limit
\begin{align*}
\lim_{\eps\to 0}\sqrt{\log(\eps^{-1})}\nu_{\gamma,\eps}(dx):=\lim_{\eps\to 0}\sqrt{\log(\eps^{-1})}e^{\gamma G_\eps(x)-\frac{\gamma^2}{2}\E[(G_\eps(x))^2]}\dx.
\end{align*} 
The deterministic renormalization converges weakly in probability on any compact $K\subset D$ for the class of general log-correlated Gaussian fields $G$ that have the decomposition $G=S+H$ with $S$ a $\star$-scale invariant field and $H$ an a.s. Hölder-continuous Gaussian field\cite[Theorem 5.3]{JuSaWe19a}. Recall that the decomposition requires that $g_G\in H_{loc}^{d+\delta}$ for all $\delta>0$. If the sub-critical GMC is constructed on $D$, then we work with $K\subset D$ in the critical case. At this stage we refer the reader to the review \cite{Po20a} and references therein for more details of the critical case.

The result that we need for the GMC measure corresponding to a $\star$-scale invariant field $S$ is the following.

\begin{theorem}
\label{thm: Existence and moments of star scale invariant GMC}
Let $S_\eps$ be the cut-off approximation of a $\star$-scale invariant Gaussian field, $\gamma\in[0,\sqrt{2d})$ and $f\in C_b(D)$. Then 
\begin{align*}
\int_{D}f(x)e^{\gamma S_\eps(x)-\frac{\gamma^2}{2}\E[(S_\eps(x))^2]}\dx:=\nu_{\gamma,\eps}(f) \overset{\eps\to 0}{\longrightarrow}\nu_{\gamma,S}(f)
\end{align*}
for all $f$ almost surely, that is, there exists a random measure $\nu_{\gamma,S}$ such that the measures $\nu_{\gamma,\eps}$ converge almost surely in the weak topology of the space of Radon measures supported on $D$. There also exists $q_c=q_c(\gamma)$ such that for every $q\in(-\infty,q_c)$ and $f\in C_c(D)$ with $f\geq 0$ and $\int f \dx >0$ we have
\begin{align*}
\E\abr{(\nu_\gamma(f))^q}<\infty.
\end{align*} 
Furthermore, if $\gamma=\gamma_c$, then we have
\begin{align*}
\sqrt{\log(\eps^{-1})}\mu_{\gamma_c,\eps}(f)\to\nu_{\gamma,S}(f)
\end{align*} 
in probability and the threshold for the existence of the moments is $q_c(\gamma_c)=1$.
\end{theorem}  
\noindent
The existence of the almost sure limit for the sub-critical case is proven in \cite{Ka85a, AlRhVa13a}. The existence of the positive moments for $q_c=2d/\gamma^2$ in the sub-critical case goes back to Kahane \cite{Ka85a} and the existence of the negative moments in this case is proved in \cite[Proposition 3.6]{RoVa10a}. The existence of the limit and the moments, and the equivalence with the non-deterministic renormalization called the derivative martingale renormalization in the critical case are proved in \cite[Theorem 5 and Corollary 6]{DuRhSh14a}.

The GMC measures can be constructed with respect to more general reference measures than just the Lebesgue measure. Below we will briefly discuss this more general framework. First we need the following definition.
\begin{definition}[$\alpha$-energy of a Radon measure]
\smallskip
\noindent
Let $D\subset \R^d$ be a bounded domain and $\mu$ a Radon measure on $D$. Then we call the intergral
\begin{align*}
\mathcal{E}_\mu(\alpha):=\int_{D\times D}|x-y|^{-\alpha}\mu\otimes\mu(\dx,\dy).
\end{align*}
the $\alpha$-energy of the measure $\mu$. Furthermore, we denote by $\tilde{d}(\mu)$ the supremum of all such $\alpha$ that $\Ecal_\mu(\alpha)<\infty$ 
\end{definition}

\begin{remark}
Note that for the Lebesque measure $\tilde{d}=d$, that is, the dimension of Euclidean space.
\end{remark}

In \cite{RoVa10a} and \cite{Be17a} it has been shown that for $\gamma<\sqrt{2\tilde{d}(\mu)}$ and a Radon measure $\mu$, with the generalized\footnote{General here meaning that we can take a convolution with respect to a Radon measure with certain properties, but not necessarily having a density with respect to the Lebesgue measure.}  convolution approximations of log-correlated fields $G_\eps$ the measures 
\begin{align*}
\nu_{\gamma,\eps,\mu}(\dx ) :=\exp(\gamma G_{\eps}(x)-(\gamma^2/2)\E[(G_\eps(x))^2])\mu(\dx )
\end{align*} 
converge in the space of Radon measures on $D$ with respect to the weak topology of convergence of measures and the limit is non-trivial and itself is a random Radon measure. Moreover, the limit is independent of the convolution kernel. More general uniqueness results are given in \cite{JuSa17a} and \cite{Sh16a}. In \cite{RoVa10a} the stochastic convergence is in law and in \cite{Be17a} in probability. 

Before discussing the scaling properties of GMC measures associated to $\star$-scale invariant fields, we state the following useful lemma about the convergence of random measures multiplied by random continuous functions.

\begin{lemma}
\label{lemma: Convergence of random measures multiplied by a random function}
Let $\{\mu_n\}_{n\in\N}$ be a sequence of locally finite random measures on a bounded domain $D\subset\R^d$ converging a.s. to a locally finite random measure $\mu$ in the following sense:
let $K\subset D$ be compact, then
\begin{align*}
\P\rbr{\mu_n(f)\to\mu(f), \,\, \forall f\in C(K)}=1.
\end{align*} 
Also, suppose that $\{g_n\}_{n\in\N}$ is a sequence of a.s. continuous random functions such that
\begin{align*}
\P\rbr{g_n \to g \text{ in } C(K)}=1.
\end{align*}
Then 
\begin{align*}
\P\rbr{\mu_n(g_nf)\to\mu(gf),\,\, \forall f\in C(K)}=1.
\end{align*}
\end{lemma}
Furthermore, as noted after \Cref{thm:field from the measure general log correlated Gaussian case} our assumptions guarantee that we are always working on some compact subset of $D$ so that we can use this lemma. 

\begin{proof}
Fix $f\in  C(K)$. Then we have 
\begin{align*}
|\mu_n(g_nf)-\mu(gf)|&\leq |\mu_n([g_n-g]f)|+|(\mu_n-\mu)(gf)|
\\
&\leq \norm{f}_\infty\norm{g_n-g}_{\infty}\mu_n(K)
\\
&\quad+|(\mu_n-\mu)(gf)|.
\\
&<\eps
\end{align*}
for any $\eps>0$ and large enough $n$. Note that $\mu_n(K)$ converges by taking $f=1\in C(K)$ and noting that converging real sequence is bounded. Thus, for a fixed test function $f$, we have $\mu_n(g_nf)\to \mu(gf)$ almost surely. Next we use separability of $C(K)$. There exists a countable dense subset $\Fcal\subset C(K)$. By countability, we have
\begin{align*}
\P\rbr{\mu_n(g_nf)\to \mu (gf),\,\,\forall f\in \Fcal}=1.
\end{align*}
Then given $f\in C(K)$ let us pick $\{f_k\}_{k\in\N}\subset \Fcal$ such that $f_k\to f$ in $C(K)$ as $k\to \infty$. We have
\begin{align*}
|\mu_n(g_nf)-\mu(gf)|&\leq |\mu_n(g_nf)-\mu_n(g_nf_k)|+|\mu_n(g_nf_k)-\mu(gf_k)|
\\
&\quad+|\mu(gf_k)-\mu(gf)|
\\
&\leq \|f-f_k\|_\infty(\sup_n(\|g_n\|_\infty)\sup_n(\mu_n(K))+\|g\|_\infty\mu(K)))+|\mu_n(g_nf_k)-\mu(gf_k)|,
\end{align*} 
Given $\eps>0$ we can always first choose $k$ such that the first term is less than $\eps/2$ and then $n\equiv n(k)$ so large that the second term is also less than $\eps/2$ almost surely by the first part of the proof. This implies the result. 
\end{proof}

\subsubsection{Properties of the $\star$-scale invariant GMC measures}
\label{sec:Properties of multiplicative chaos}
The properties of the $\star$-scale invariant GMC measure presented in this section are no doubt known or at least not surprising to an expert in this field, but in the absence of a reference in the form we need them, we also present proofs. In this section, we will omit the index corresponding to the underlying field from the GMC measure. The scaling properties of the fields in  \Cref{sec:Scaling properties of star-scale invariant fields} translate to the following properties on the level of GMC measures. Let $K\subset D$ be compact and such that $\dist(K,\partial D)\geq \eps_0$ for some $\eps_0\in (0,1)$, $\eta\in C_c(\R^d)$ with $\supp(\eta)\subset B_1(0)$, $\eps\in (0,\eps_0)$  and denote $\eta_{\eps}(x)=\eps^{-d}\eta(x/\eps)$. Furthermore, define $C_\gamma(\delta)=1$ if $\gamma<\sqrt{2d}$ and $C_\gamma(\delta)=\sqrt{\log(\delta^{-1})}$ if $\gamma=\gamma_c=\sqrt{2d}$. Then we have for $x\in K$ the following fundamental representation that turns out to be crucial in the proofs of our main results:
\begin{align}
\label{eq: definition of nuGammaEpsx}
\int_{\R^d}&\eta_\eps(y-x)\nu_{\gamma}(\dy )\nonumber
\\
=&\lim_{\delta\to 0}C_\gamma(\delta)\int_{\R^d}\eta_\eps(y-x)e^{\gamma S_{\delta}(y)-\frac{\gamma^2}{2}\log(\delta^{-1})}\dy \nonumber
\\
=&\lim_{\delta\to 0}C_\gamma(\delta)\int_{\R^d}\eta(u)e^{\gamma S_\delta(x+\eps u)-\frac{\gamma^2}{2}\log(\delta^{-1})}\du  \nonumber
\\
=&e^{\gamma S_\eps(x)-\frac{\gamma^2}{2}\log(\eps^{-1})}\nonumber
\\
&\quad\times\lim_{\delta\to 0}C_\gamma(\delta)\int_{\R^d} \eta(u)e^{\gamma [S_\eps(x+\eps u)-S_\eps(x)]}e^{\gamma [S_\delta(x+\eps u)-S_\eps(x+\eps u)]-\frac{\gamma^2}{2}\log(\abr{\frac{\delta}{\eps}}^{-1})}\du  \nonumber
\\
=&e^{\gamma S_\eps(x)-\frac{\gamma^2}{2}\log(\eps^{-1})}\lim_{\delta\to 0}C_\gamma(\delta/\eps)(1+o(1))\int_{\R^d}\eta(u)e^{\gamma Y_{\eps,x}(u)}e^{\gamma Z_{(\eps,\delta,x)}(u)-\frac{\gamma^2}{2}\E[(Z_{(\eps,\delta,x)}(u))^2]}\du 
\nonumber
\\
:=&e^{\gamma S_\eps(x)-\frac{\gamma^2}{2}\log(\eps^{-1})}\int_{\R^d} \eta(u)e^{\gamma Y_{\eps,x}(u)}\nu_{\gamma}^{\eps, x}(\du),
\end{align}
where we have used the auxiliary fields $Y_{\eps,x}$ and $Z_{(\eps,\delta,x)}$ defined in \Cref{eq:Def Y,eq: Def Z} respectively. The limit $\delta \to 0$  exists for fixed $\eps$ by \Cref{lemma: Convergence of random measures multiplied by a random function} above since $\delta/\eps\to 0$ as $\delta\to 0$. 
Note that if $\gamma=\gamma_c$,
to use \Cref{lemma: Convergence of random measures multiplied by a random function} we need to work with some subsequence since the limit constructing the critical GMC measure only exists in probability. The measure $\nu_\gamma^{\eps,x}$ is independent of $S_\eps$ and $Y_{\eps,x}$, and distributed like $\nu_{\gamma,B_1(0)}$, that is, a GMC measure corresponding to $S$ restricted to the unit ball. This follows directly from properties 1. and 2. in \Cref{prop: properties of the process Z}. Indeed, at the level of approximations, $\nu_{\gamma}^{\eps,x}(\du )$ is given by
\begin{align*}
C_\gamma(\delta/\epsilon)e^{\gamma Z_{(\eps,\delta,x)}(u)-\frac{\gamma^2}{2}\E[(Z_{(\eps,\delta,x)}(u))^2]}\du .
\end{align*}
This is independent of $S_\eps$, because $Z_{(\eps,\delta,x)}$ is, and it is distributed like
\begin{align*}
C_\gamma(\delta/\eps)e^{\gamma S_{\frac{\delta}{\eps}}(u)-\frac{\gamma^2}{2}\E[(S_{\frac{\delta}{\eps}}(u))^2]}\du ,
\end{align*}
which is again distributed like $\nu_{\gamma, B_1(0)}$ 
for any fixed $\eps>0$ in the limit $\delta\to 0$. The fact that $Z_{(\eps,\delta,x)}$ is distributed like $S_{\delta/\eps}$ also dictates that the correct renormalization constant above is $C_{\gamma}(\delta/\eps)$ instead of just $C_\gamma(\delta)$. Note that if $\supp(\eta_\eps(\cdot-x))\subset D$, where $D$ is such that $\nu_\gamma$ has support in $D$, then under the assumptions $\nu_\gamma^{\eps,x}$ is well-defined on $B_1(0)$.

Similarly, we could do this simultaneously for two fixed points of $x,x'\in K$. Then 
\begin{align*}
\bigg( \int_{\R^d} \eta_\eps(x-y)\nu_\gamma &(\dy), \int_{\R^d} \eta_\eps(x'-y)\nu_\gamma (\dy)\bigg)
\\
&=\bigg(e^{\gamma S_\eps(x)-\frac{\gamma^2}{2}\log(\eps^{-1})}\int_{\R^d} \eta(u)e^{\gamma Y_{\eps,x}(u)}\nu_{\gamma}^{\eps, x}(\du),
\\
&\quad \quad e^{\gamma S_\eps(x')-\frac{\gamma^2}{2}\log(\eps^{-1})}\int_{\R^d} \eta(u)e^{\gamma Y_{\eps, x'}(u)}\nu_{\gamma}^{\eps, x'}(\du)\bigg),
\end{align*}
where both $\nu_\gamma^{\eps,x}$ and $\nu_\gamma^{\eps,x'}$ are distributed like $\nu_{\gamma,B_1(0)}$ and independent of $S_\eps$. Furthermore, if $x$ and $x'$ are not too close, then $\nu_\gamma^{\eps,x}|_{\supp(\eta)}$ and $\nu_\gamma^{\eps,x}|_{\supp(\eta)}$ are independent. More precisely, we have the following proposition.

\begin{proposition}
\label{prop: independence of nu eps x and nu eps x'}
Let $K\subset D$ compact and satisfying $\dist(K,\partial D)\geq \eps_0$ for some $\eps_0\in (0,1)$. Also let $\eps\in (0,\eps_0)$ be small enough so that $\supp(\eta_\eps(\cdot-w))\subset D$ for $w=x,x'$ with $w\in K$ and $|x-x'|\geq 3\eps$. Then the restrictions of $\nu_\gamma^{\eps,x}$ and $\nu_\gamma^{\eps,x'}$ to $\supp(\eta)\subset B_1(0)$ are independent. $\nu_\gamma^{\eps,x}$ was defined in \Cref{eq: definition of nuGammaEpsx}.
\end{proposition}

\begin{proof}
Let $f,g\in C_c(B_1(0))$ or equivalently we could take $f,g\in C_c(\supp(\eta))$ and look at $\nu_\gamma^{\eps,x}(f)$ and $\nu_\gamma^{\eps,x}(g)$ and we do so to keep the notation simple. These in turn are independent if the approximations are. We define
\begin{align*}
\nu_{\gamma,\delta}^{\eps,x}(f):=C_\gamma(\delta/\eps)\int_{\R^d}f(u)e^{\gamma Z_{(\eps,\delta,x)}(u)-\frac{\gamma^2}{2}\E[(Z_{\eps,\delta,x}(u))^2]}\du.
\end{align*} 

Next note that the fields $Z_{(\eps,\delta,x)}$ and $Z_{(\eps,\delta,x')}$ are independent. Indeed, we have
\begin{align*}
\Cov(Z_{(\eps,\delta,x)}(u),Z_{(\eps,\delta,x')}(v))&=K_{\delta}(x-\eps u,x'-\eps v) 
\\
&\quad+ K_{\eps}(x-\eps u,x'-\eps v)
\\
&\quad-K_{(\delta,\eps)}(x-\eps u,x'-\eps v) 
\\
&\quad - K_{(\eps,\delta)}(x-\eps u,x'-\eps v).
\end{align*}
Noting that $\delta<\eps$ the above becomes
\begin{align*}
\int_{\eps^{-1}}^{\delta^{-1}}k_1(s|x-x'+\eps(u-v)|)\frac{\ds}{s}\overset{s'=\eps s}{=}\int_1^{\frac{\eps}{\delta}}k_1\left(s'\left|\frac{x-x'}{\eps}+u-v\right|\right)\frac{\ds'}{s'}.
\end{align*}
Now $s>1$ and $|\eps^{-1}(x-x')+(u-v)|\geq \eps^{-1}|x-x'|-|u-v|\geq 3-(|u|-|v|)>1$ so that the covariance is zero by the fact that $\supp(k_1)\subset [0,1]$. Since the fields $Z_{(\eps,\delta,x)}$ and $Z_{(\eps,\delta,x')}$ are independent under the assumption, also $\nu_{\gamma,\delta}^{\eps,x}(f)$ and $\nu_{\gamma,\delta}^{\eps,x'}(g)$ are. 
\end{proof}

\section{Proofs of the main results}
\label{sec:Proofs of the main results}
In this section we present the proofs of the main theorems. To prove \Cref{thm:field from the measure general log correlated Gaussian case} we need similar result for the $\star$-scale invariant fields $S$. These fields were defined in \Cref{sec:star scale invariant fields}. We will state it here and prove it at the end of this section after we have proved our main theorems assuming it.

\begin{proposition}
\label{thm:field from the measure star scale invariant case}
Let $\nu_{\gamma,S}$ be a GMC measure constructed from a $\star$-scale invariant Gaussian field $S$. Let $\eta\in C_c(\R^d)$ with $\eta\geq 0$, $\int\eta dx=1$ and $\supp(\eta)\subset B_1(0)$, and $\gamma\in [0,\sqrt{2d}]$. Then there exists $\eps_{0}\in (0,1)$ such that for $0<\eps<\eps_0$ there exists a deterministic function $F_{\gamma,\eps,\eta}(x)$ such that for any $\psi \in C_c^\infty(D)$ with $\dist(\supp(\psi),\partial D)>\eps_0$ we have
\begin{align*}
\int_{D}\psi(x)\rbr{\frac{1}{\gamma}\log\abr{\int_{\R^d}\eta_{\eps}(y-x)\nu_{\gamma,S}(\dy )}-F_{\gamma,\eps,\eta}(x)}dx\overset{\eps\to 0}{\longrightarrow} \innerp{S}{\psi}
\end{align*}
in probability. Above $\eta_\eps(\cdot):=\eps^{-d}\eta(\cdot /\eps)$. 

\end{proposition}

\subsection{Proofs of \Cref{thm:field from the measure general log correlated Gaussian case,thm:field from the measure for mildly non-Gaussian case} assuming \Cref{thm:field from the measure star scale invariant case}}

\label{sec:Proofs of theorems A and B assuming prop 31}
Given \Cref{thm:field from the measure star scale invariant case}, we can now prove \Cref{thm:field from the measure general log correlated Gaussian case}. Our main tool is the decomposition $G=S+H$ mentioned in the introduction and established in \cite[Theorem A]{JuSaWe19a}. Recall that $S$ is a $\star$-scale invariant field and $H$ is an almost surely Hölder-continuous Gaussian field defined in the same probability space as $S$. Recall also that the fields of the decomposition are not necessarily independent. 
\begin{proof}[Proof of \Cref{thm:field from the measure general log correlated Gaussian case}]
We begin with the following lemma.

\begin{lemma}
\label{lemma:the absolute conituity of muX and muY}
Let $G$ be a log-correlated Gaussian field with the decomposition $G=S+H$ described above. Also let $f\in C_c(D)$. Then, the corresponding GMC measure satisfies
\begin{align*}
\nu_{\gamma,G}(f)=\int_{\R^d}f(u)e^{\gamma H(u)-\frac{\gamma^2}{2}[g_G(u,u)-g_S(u,u)]}\nu_{\gamma,S}(\du )
\end{align*} 
\end{lemma}

\begin{proof}[Proof of \Cref{lemma:the absolute conituity of muX and muY}]
Let $C_\gamma(\delta)=1$ for $\gamma\in[0,\sqrt{2d})$ and $C_{\gamma_c}=\sqrt{\log(\delta^{-1})}$. Then we have for any $f\in C_c(D)$ 
\begin{align*}
\nu_{\gamma,G}(f)=&\lim_{\delta\searrow 0}C_\gamma(\delta)\int_{\R^d}f(u)e^{\gamma G_\delta(u)-\frac{\gamma^2}{2}\E[G_\delta(u)^2]}\du 
\\
=&\lim_{\delta\searrow 0}C_\gamma(\delta)\int_{\R^d}\big(f(u)e^{\gamma H_\delta(u)}
\\
&\qquad\qquad \times e^{\frac{\gamma^2}{2}[\E[(S_\delta(u))^2]-\E[(G_\delta(u))^2]]}e^{\gamma S_\delta(u)-\frac{\gamma^2}{2}\E[(S_\delta(u))^2]}\big)\du 
\\
=&\int_{\R^d}f(u)e^{\gamma H(u)-\frac{\gamma^2}{2}(g_G(u,u)-g_S(u,u))}\nu_{\gamma,S}(\du ),
\end{align*}
where the limit in the last step is taken in such a way (along a subsequence if necessary) that \Cref{lemma: Convergence of random measures multiplied by a random function} applies. Above we use the convolution approximations with the same mollifier for the variances $\E[(S_\eps(x))^2]$ and $\E[(G_\eps(x))^2]$ so that the logarithmic parts completely cancel.  
\end{proof}

Using \Cref{lemma:the absolute conituity of muX and muY} we can now write
\begin{align*}
&\log(\int_{\R^d}\eta_\eps(y-x)\nu_{\gamma,G}(\dy ))
\\
&=\log\bigg(\int_{\R^d}\big(\eta_\eps(y-x)e^{\gamma [H(y)+H(x)-H(x)]}
\\
&\qquad\qquad\quad \times e^{-\frac{\gamma^2}{2}[g_G(y,y)-g_S(y,y)+g_G(x,x)-g_S(x,x)-(g_G(x,x)-g_S(x,x))]}\big)\nu_{\gamma,S}(\dy )\bigg)
\\
&\quad+\log(\int_{\R^d}\eta_\eps(y-x)\nu_{\gamma,S}(\dy ))-\log(\int_{\R^d}\eta_\eps(y-x)\nu_{\gamma,S}(\dy ))
\\
&=\gamma H(x)-\frac{\gamma^2}{2}[g_G(x,x)-g_S(x,x)]+\log(\int_{\R^d}\eta_\eps(y-x)\nu_{\gamma,S}(\dy ))
\\
&\quad+\underbrace{\log(\frac{\int_{\R^d}\eta_\eps(y-x)e^{\gamma [H(y)-H(x)]-\frac{\gamma^2}{2}[g_G(y,y)-g_S(y,y)-(g_G(x,x)-g_S(x,x))]}\nu_{\gamma,S}(\dy )}{\int_{\R^d}\eta_\eps(y-x)\nu_{\gamma,S}(\dy )})}_{:=R_\eps(x)}
\end{align*}

In the following, we include the subscript for the field in the counter term to emphasise that the counter terms differ for different fields. Next assuming that the limit below exists we can use \Cref{thm:field from the measure star scale invariant case} to write
\begin{align*}
&\int_{D}\psi(x)\abr{\frac{1}{\gamma}\log(\int_{\R^d}\eta_\eps(y-x)\nu_{\gamma,G}(\dy))-F_{\gamma,\eps,\eta,S}(x)}\dx 
\\
&\longrightarrow \innerp{S}{\psi}+\int_{D}\psi(x)H(x)\dx
-\frac{\gamma}{2}\int_{D}\psi(x)[g_G(x,x)-g_S(x,x)]\dx 
\\
&\qquad\qquad+\frac{1}{\gamma}\lim_{\eps\searrow 0}\int_{D}\psi(x)R_\eps(x)\dx .
\end{align*}
From this we see that to show the existence, it is enough to show that the last term vanishes. If we assume this we can define
\[F_{\gamma,\eps,\eta,G}(x):=F_{\gamma,\eps,\eta,S}(x)+\frac{\gamma}{2}(g_S(x,x)-g_G(x,x))\]
and note that
\begin{align*}
\innerp{S}{\psi}+\int_{D}\psi(x)H(x)\dx =\innerp{G}{\psi}.
\end{align*} 
Thus, we are done if we show that 
\begin{align*}
\lim_{\eps\to 0}\int_D\psi(x)R_\eps(x)dx=0.
\end{align*}

First define 
\begin{align*}
M(x,y):=&\gamma[H(x)-H(y)]
\\
&\quad+(\gamma^2/2)[g_G(x,x)-g_S(x,x)-(g_G(y,y)-g_S(y,y))]. 
\end{align*}
Obviously, we have 
\begin{align*}
e^{-\sup_{u\in \supp(\eta_\eps(\cdot-x))} |M(x,u)|}&\leq e^{\inf_{u\in \supp(\eta_\eps(\cdot-x))}M(x,u)}\leq e^{M(x,y)}
\\
&\leq e^{\sup_{u\in \supp(\eta_\eps(\cdot-x))}M(x,u)}\leq e^{\sup_{u\in \supp(\eta_\eps(\cdot-x))}|M(x,u)|}
\end{align*}
for all $y\in\supp(\eta_\eps(\cdot-x))\subset B_\eps(x)$.
By monotonicity of the logarithm, we have
\begin{align*}
&-\sup_{u\in B_\eps(x)}|M(x,u)|\leq R_\eps(x)\leq\sup_{u\in B_\eps(x)}|M(x,u)|
\\
& \qquad \Rightarrow \quad |R_\eps(x)|\leq \sup_{u\in B_\eps(x)}|M(x,u)|
\end{align*}
Then along any deterministic sequence $\eps_n\searrow 0$ as $n\to \infty$ we have 
\begin{align*}
\lim_{n\to\infty}\sup_{\substack{x\in\supp(\psi) \\ y\in B_{\eps_n(x)}}}|M(x,y)|=0.
\end{align*}
The above holds, because for any $x\in\supp(\psi)$ and any $n\in\N$,  $H,g_G(\cdot,\cdot)$ and $g_S(\cdot,\cdot)$ are uniformly continuous in $B_{\eps_n}(x)$. Thus, the sequence of functions 
\begin{align*}
g_n(x):=\sup_{y\in B_{\eps_n}(x)}|M(x,y)|
\end{align*}
converges uniformly and we obtain 
\begin{align*}
\lim_{\eps\to\infty}\bigg|\int_{D}\psi(x)R_\eps(x)\dx \bigg|&\leq\lim_{\eps\to 0}\int_{D}|\psi(x)|\sup_{y\in B_\eps(x)}|M(x,y)|\dx 
\\
&=\int_{D}|\psi(x)|\lim_{n\to\infty}g_n(x)\dx =0.
\end{align*} 
\end{proof}

Given \Cref{thm:field from the measure general log correlated Gaussian case} we can now prove \Cref{thm:field from the measure for mildly non-Gaussian case}. The strategy is similar to the proof of \Cref{thm:field from the measure general log correlated Gaussian case} given \Cref{thm:field from the measure star scale invariant case}.
\begin{proof}[Proof of \Cref{thm:field from the measure for mildly non-Gaussian case}]
We write analogously to the proof of \Cref{thm:field from the measure general log correlated Gaussian case} 
\begin{align*}
\log\bigg(\int_{\R^d}&\eta_\eps(y-x)\nu_{\gamma,X}(\dy )\bigg)
\\
&=\log(\int_{\R^d}\eta_\eps(y-x)e^{\gamma H(y)}\nu_{\gamma,G}(\dy ))
\\
&=\gamma H(x)+\log(\int_{\R^d}\eta_\eps(y-x)\nu_{\gamma,G}(\dy ))
\\
&\quad+\log(\frac{\int_{\R^d}\eta_\eps(y-x)e^{\gamma H(y)-H(x)}\nu_{\gamma,G}(\dy )}{\int_{\R^d}\eta_\eps(y-x)\nu_{\gamma,G}(\dy )}).
\end{align*}
and proceed exactly as we did with the proof of \Cref{thm:field from the measure general log correlated Gaussian case}, but replace $S$ everywhere with $G$ and drop the terms $g_G-g_S$. Thus, we are done with proof the \Cref{thm:field from the measure for mildly non-Gaussian case}.
\end{proof}

\subsection{Discussion about using \Cref{def: non gaussian multiplicative chaos 2} in \Cref{thm:field from the measure for mildly non-Gaussian case}}
\label{subsec:Discussion about the other possible definition of non-Gaussian multiplicative chaos in theorem C} In this section, we discuss using the alternative \Cref{def: non gaussian multiplicative chaos 2} for $\nu_{\gamma,X}$ in \Cref{thm:field from the measure for mildly non-Gaussian case}. First of all, we assume that the  limit in \Cref{def: non gaussian multiplicative chaos 2} exists almost surely (possibly along some subsequence) in the weak topology of convergence of measures\footnote{This is the topology, that is used in our example discussed in \ref{sec:Applications}. However, convergence in the vague topology would suffice for us since we are always working on a compact $K\in D$. The difference in the definitions of weak and vague topology is replacing $f\in C_b(D)$ with $f\in C_c(D)$ and in practice the main difference is that in the vague topology the total mass does not need to converge.}
Then we have at least along some subsequence
\begin{align*}
\lim_{\delta\to 0}\int_Df(x)&e^{\gamma X_\delta(x)}(\E[e^{\gamma X_\delta(x)}])^{-1}\dx 
\\
&=\lim_{\delta\to 0}\int_{D}\big( f(x)e^{\gamma H_\delta(x)}(\E[e^{\gamma X_\delta(x)}])^{-1}
\\
&\qquad\qquad\times e^{\frac{\gamma^2}{2}\E[ (G_\delta(x))^2]}e^{\gamma G_\delta(x)-\frac{\gamma^2}{2}\E[(G_\delta(x))^2]}\big)\dx 
\\
&=\lim_{\delta \to 0}\int_Df(x)e^{\gamma H_\delta(x)}\frac{\E[e^{\gamma G_\delta(x)}]}{\E[e^{\gamma X_\delta(x)}]}e^{\gamma G_\delta(x)-\frac{\gamma^2}{2}\E[(G_\delta(x))^2]}\dx 
\\
&:=\int_D f(x)e^{\gamma H(x)}g(x)\nu_{\gamma,G}(\dx)
\end{align*}
provided that the limit 
\begin{align}
\label{eq: def function g}
e^{\gamma H(x)}g(x):=\lim_{\delta\to 0}e^{\gamma H_\delta(x)}g_\delta(x):=\lim_{\delta\to 0}e^{\gamma H_\delta(x)}\frac{\E[e^{\gamma G_{\delta}(x)}]}{\E[e^{\gamma X_\delta(x)}]}
\end{align}
satisfies the assumptions in \Cref{lemma: Convergence of random measures multiplied by a random function}.
\begin{remark}
The case that $G$ and $H$ are independent gives us the minimal requirement that the exponential moments of $H$ need to exist and be bounded from below, 
because then Gaussian part cancels out.
\end{remark}
Next we would like to proceed again as in the proof of \Cref{thm:field from the measure general log correlated Gaussian case}. We can write
\begin{align}
\label{eq:def of R epsilon}
\log\bigg(\int_{\R^d}&\eta_\eps(y-x)\nu_{\gamma,X}(\dy )\bigg)\nonumber
\\
&=\log(\int_{\R^d}\eta_\eps(y-x)e^{\gamma H(y)}g(x)\nu_{\gamma,G}(\dy )) \nonumber
\\
&=\gamma H(x)+\log(\int_{\R^d}\eta_\eps(y-x)\nu_{\gamma,G}(\dy ))\nonumber
\\
&\quad+\underbrace{\log(\frac{\int_{\R^d}\eta_\eps(y-x)e^{\gamma[H(y)-H(x)]}g(y)\nu_{\gamma,G}(\dy )}{\log(\int_{\R^d}\eta_\eps(y-x)\nu_{\gamma,G}(\dy ))})}_{:=R_\eps(x)}.
\end{align}
If we now assume that the limit $\lim_{\eps\to 0}R_\eps(x)$ defines a reasonably nice deterministic function, continuous for example, we can again mimic the proof of  \Cref{thm:field from the measure general log correlated Gaussian case}. More precisely, we need to show that $R_\eps(x):=R(x)+\Delta_\eps(x)$ with $R$ being deterministic and $\Delta_\eps(x)\to 0$ as $\eps\to 0$ pointwise with local uniform bounds.  Indeed, if we define
\begin{align*}
F_{\gamma,\eps,\eta,X}=F_{\gamma,\eps,\eta,G}+\frac{1}{\gamma}R
\end{align*}
we obtain by \Cref{thm:field from the measure general log correlated Gaussian case}
\begin{align*}
\lim_{\eps\to 0}\int_{D}\psi(x)&\rbr{\frac{1}{\gamma}\log(\int_{\R^d}\eta_\eps(y-x)\nu_{\gamma,X}(\dy ))-F_{\gamma,\eps,\eta,X}(x)}\dx 
\\
&=\innerp{G}{\psi}+\int_{D}\psi(x)H(x)\dx
\\
&=\innerp{X}{\psi}.
\end{align*}
Next we claim that if we assume $g>0$ everywhere, then the above assumption holds with $R(x)=\log(g(x))$. Indeed, we can write $g(y)=g(x)(1+[g(y)-g(x)]/g(x)))$. By this decomposition, we have
\begin{align*}
R_\eps(x)&=\log(g(x))
\\
&\quad+\log(\frac{\int_{\R^d}\eta_\eps(y-x)e^{\gamma(H(y)-H(x))}\rbr{1+\frac{g(y)-g(x)}{g(x)}}\nu_{\gamma,S}(\dy )}{\int_{\R^d}\eta_\eps(y-x)\nu_{\gamma,S}(\dy )})
\end{align*}
Now we proceed as in the proof of \Cref{thm:field from the measure general log correlated Gaussian case} and define $M(x,y)=\gamma[H(y)-H(x)]+\log(1+[g(y)-g(x)]/g(x)))$. Then we need to prove that along any deterministic sequence such that $\eps_n\to 0$ we have
\begin{align*}
\lim_{\eps_n \to 0}\sup_{\substack{x\in \supp(\psi) \\ y\in B_{\eps_n}(x)}}|M(x,y)|=0.
\end{align*}
This is clear by uniform continuity on compact sets since we assumed that $g(x)\neq 0$ for all $x\in \supp(\psi)$.

We can now collect the assumptions that we need in order to prove \Cref{thm:field from the measure for mildly non-Gaussian case} also using \Cref{def: non gaussian multiplicative chaos 2}.
\begin{enumerate}
\item The limit in \Cref{def: non gaussian multiplicative chaos 2} exists in the weak topology of Radon measures with support in $D$ and almost surely (possibly along a subsequence).
 
\item The limit in \Cref{eq: def function g} satisfies the assumptions in \Cref{lemma: Convergence of random measures multiplied by a random function} and the limiting function $g$ is strictly positive everywhere.
\end{enumerate} 
\begin{remark}
We do not have any easily verifiable conditions to guarantee that these assumptions are  satisfied. Furthermore, there are no canonical general approximations to be used in \Cref{def: non gaussian multiplicative chaos 2} and in the limit in \Cref{eq: def function g}. It is natural to believe that the approximations that are used to construct the specific field $X$ and to prove the coupling to a log-correlated Gaussian field, would be the easiest to use in order to derive suitable estimates. 
\end{remark}
Next we give an illustrative example to show that these assumptions are non-trivial and may not be satisfied even if $G$ and $H$ both are Gaussian in the decomposition $X=G+H$. Indeed, we take $G$ to be defined by the following series as in \Cref{eq:series expansion for log Gaussian fields} 
\begin{align*}
G(\theta):=\sum_{n=1}^\infty\frac{1}{\sqrt{n}}(A_n\sin(n\theta)+B_n\cos(n\theta)),
\end{align*}
where $\theta\in [0,2\pi)$ and $\{A_n\}$ and $\{B_n\}$ are two independent i.i.d  sequences of standard Gaussian random variables and the convergence is in suitable distributional sense for example in some negative exponent Sobolev space. This corresponds to a log-correlated Gaussian field on the unit circle $C_1$. In fact, this is the restriction of the  2-dimensional Gaussian free field to the unit circle $C_1=\{z\in\C\mid |z|=1\}$ up to arbitrary additive constant since the whole plane GFF is itself only defined up to an additive constant. This field is constructed for example in \cite[section 6]{JuSa17a}. Then, let 
\begin{align*}
H=\sum_{n=1}^\infty \frac{B_n}{\sqrt{n}\log(n)}
\end{align*} 
so that $H$ is a constant in $\theta$ and thus even $C^{\infty}(D)$ field. $H$ is also clearly Gaussian as a sum of Gaussian random variables. Also, the series for $H$ converges in $L^2$. Now let $G_N$ and $H_N$ be the approximations, where we cut off both series at $n=N$. In this case, we have at $\theta=0$ 
\begin{align*}
\frac{\E\abr{e^{\gamma G_N(0)}}}{\E\abr{e^{\gamma (G_N(0)+H_N)}}}&=\frac{\prod_{n=1}^N\E\abr{e^{\frac{\gamma}{\sqrt{n}}B_n}}}{\prod_{n=1}^N\E\abr{e^{\frac{\gamma}{\sqrt{n}}B_n+\frac{\gamma}{\sqrt{n}\log(n)}B_n}}}
\\
&=\prod_{n=1}^N e^{-\frac{\gamma^2}{n\log(n)}-\frac{\gamma^2}{2n(\log(n))^2}}
\\
&=e^{-\gamma^2\sum_{n=1}^N \frac{1}{n\log(n)}}e^{-\frac{\gamma^2}{2}\sum_{n=1}^N\frac{1}{n(\log(n))^2}}\overset{N\to \infty}{\longrightarrow}0,
\end{align*} 
because the first sum in the exponents diverges and the second converges. 
Thus, the function $g$ defined in \ref{eq: def function g} corresponding to the field $X=G+H$ would not be strictly positive. Furthermore, let us compute the covariance of $X_N=G_N+H_N$
\begin{align*}
\E[X_N(\theta)X_N(\phi)]=\E[G_N(\theta)G_N(\phi)]+\E[(G_N(\theta)+G_N(\phi))H_N]+\E[H_N^2],
\end{align*}
where we have used the fact that $H_N$ is constant in $\theta$ and $\phi$. In \cite[section 6]{Ka85a} the authors show that 
\begin{align*}
\E[G_N(\theta)G_N(\phi)]\overset{N\to \infty}{\longrightarrow}\log(\frac{1}{2|\sin(\frac{1}{2}(\theta-\phi))|}),
\end{align*} 
which can readily be rewritten in terms of the logarithm of the  distance between two points on the unit circle $|e^{i\theta}-e^{i\phi}|$. Thus, this is the logarithmic part of the covariance of the field $X=\lim_{n\to \infty}G_N+H_N$, where the limit is taken in suitable distribution space. Furthermore,
\begin{align*}
\E[H_N^2]\to\sum_{n=1}^\infty \frac{1}{n(\log(n))^2}<\infty,
\end{align*}
but
\begin{align*}
\E[(G_N(\theta)+G_N(\phi))H_N]=\sum_{n=1}^N\frac{\cos(n\theta)+\cos(n\phi)}{n\log(n)},
\end{align*}
which does not converge as $N\to \infty$ if $\theta=0$ or $\phi=0$. Therefore, $X$ is not a log-correlated Gaussian field in the sense we introduced them in the introduction. This is, because the corresponding function $g_X(\theta,\phi)$ in the covariance kernel $C_X$ given in \ref{eq:definition of the covariance kernel of a log correlated Gaussian field} would have singularity at the origin.

\subsection{Proof of \Cref{thm:field from the measure star scale invariant case}}
\label{sec: Proof of the first main theorem}
In this section, we omit the subscript $S$ for the field in the GMC measure since there will only be GMC measures corresponding to a $\star$-scale invariant field. Then let us begin by defining what turns out to be the correct function to be subtracted
\begin{align*}
F_{\gamma,\eps,\eta}(x)&:=\frac{1}{\gamma}\E\abr{\log(\int_{\R^d}\eta(u)e^{\gamma Y_{\eps,x}(u)}\nu_\gamma^{\eps,x}(\du))}-\frac{\gamma^2}{2}\log(\eps^{-1})
\\
&:=G_{\gamma,\eps,\eta}(x)-\frac{\gamma^2}{2}\log(\eps^{-1})
\end{align*}
This expression is well-defined by the remark after \Cref{prop: independence of nu eps x and nu eps x'}. The process $Y_{\eps,x}$ was defined in \Cref{eq:Def Y}. Firstly note that this function is well-defined since it is always finite for fixed $\eps>0$. It is even bounded for $x\in\supp(\psi)$. 

\begin{proposition}
\label{prop:finiteness of the counter term}
For any $\gamma\in [0,\sqrt{2d}]$ and $\eps>0$, the function $F_{\gamma,\eps,\eta}(x)$ is bounded uniformly for $x\in \supp(\psi)$ with $\psi$ as in  \Cref{thm:field from the measure star scale invariant case}.
\end{proposition}
\noindent
For the proof of the above proposition, we need the following basic estimate
\begin{lemma}
\label{lemma:logarithmic estimate}
Let $x>0$ and $k\in\N$. Then
\begin{align*}
|\log(x)|^{k}\leq C_{a,k}(x^{a}+x^{-a})
\end{align*}
for some $C_{a,k}>0$ and all $a>0$.
\end{lemma}

\begin{proof}
Because both sides of the equation are invariant under $x\mapsto 1/x$, it is enough to concentrate on the case $x\geq 1$. Also in this case $x^{-a}\leq x^a$ for all $a>0$ so that by the positivity of both sides of $|\log(x)|\leq C_a x^a$ we can always raise both sides to the power of $k$, choose $C_{a,k}=(C_{a,1})^k$ and replace $a$ by $ka$, which also arbitrary and strictly positive for fixed $k$ if $a$ is. Thus, it is enough to show that
\begin{align*}
\log(x)\leq C_a x^a
\end{align*} 
for some $C_a$ and all $a>0$, when $x\geq 1$. Consider
\begin{align*}
f(x)=\log(x)-C_ax^a.
\end{align*} 
We have $f(1)=-C_a$ and $f'(x)=(1/x)(1-a C_a x^a)$, which is negative if $aC_a\geq 1$ i.e $C_a\geq 1/a$. This shows that $f$ is decreasing and so negative on $[1,\infty)$, which proves our claim.
\end{proof}

\begin{proof}[Proof of \Cref{prop:finiteness of the counter term}]

\smallskip
\noindent
Firstly, for $\eps>0$ it is enough to look at $G_{\gamma,\eps,\eta}$.  We have by the monotonicity of the logarithm
\begin{align*}
\frac{1}{\gamma}\E&\abr{\log(\inf_{u\in\supp(\eta)}\cbr{\exp(\gamma Y_{\eps,x}(u))}\nu_\gamma^{\eps,x}(\eta))}
\\
&\leq G_{\gamma,\eps,\eta}
\\
&\leq  \frac{1}{\gamma}\E\abr{\log(\sup_{u\in\supp(\eta)}\cbr{\exp(\gamma Y_{\eps,x}(u))}\nu_\gamma^{\eps,x}(\eta))}
\\
&\leq \E\abr{\sup_{|u|\leq 1}Y_{\eps,x}(u)}+\frac{1}{\gamma}\E[\log(\nu_\gamma^{\eps,x}(\eta))]
\end{align*}
where the lower bound can be estimated analogously. Note that by \Cref{lemma:logarithmic estimate} and \Cref{thm: Existence and moments of star scale invariant GMC}
\begin{align*}
\E[\log(\nu_\gamma^{\eps,x}(\eta))]=\E[\log(\nu_{\gamma,B_1(0)}(\eta))]<\infty,
\end{align*}
where we used the fact that $\nu_{\gamma}^{\eps,x}$ is distributed like $\nu_{\gamma,B_1(0)}$. Then denoting $R_{\gamma,\eps,\eta}=G_{\gamma,\eps,\eta}-\gamma^{-1}\E[\log(\nu_\gamma^{\eps,x}(\eta))]$ we have
\begin{align*}
|R_{\gamma,\eps,\eta}|&=\max(R_{\gamma,\eps,\eta},-R_{\gamma,\eps,\eta})
\\
&\leq\max\rbr{\E\abr{\sup_{|u|<1}Y_{\eps,x}(u)},\E\abr{-\inf_{|u|<1}Y_{\eps,x}(u)}}
\\
&\leq  \E\abr{\left|\sup_{|u|<1}Y_{\eps,x}(u)\right|}+\E\abr{\left|\inf_{|u|<1}Y_{\eps,x}(u)\right|}
\\
&<\infty
\end{align*} 
by the argument in the proof of \Cref{lemma: even moments of Y} and the fact that the absolute values of the $\inf$ and $\sup$ are identically distributed.
\end{proof}

Then we have
\begin{align*}
\E&\bigg[\bigg(\int_{D}\psi(x)\bigg\{\frac{1}{\gamma}\log(\int_{\R^d}\eta_\eps(y-x)\nu_\gamma(\dy ))\nonumber
\\
&\qquad\qquad-\abr{G_{\gamma,\eps,\eta}(x)-\frac{\gamma^2}{2}\log(\eps^{-1})}\bigg\}\dx -\int_{\R^d}\psi(x)S_\eps(x)\dx \bigg)^2\bigg]
\\
&=\E\bigg[\bigg(\int_D\psi(x)\bigg\{\frac{1}{\gamma}\log(\int_{\R^d}\eta(y)e^{\gamma Y_{\eps,x}}\nu_\gamma^{\eps,x}(\dy ))\nonumber
\\
&\qquad\qquad-\E\abr{\frac{1}{\gamma}\log(\int_{\R^d}\eta(y)e^{\gamma Y_{\eps,x}}\nu_\gamma^{\eps,x}(\dy ))}\bigg\}\dx \bigg)^2\bigg]
\\
&:=E_\eps.
\end{align*}
Above we used the fundamental representation \Cref{eq: definition of nuGammaEpsx}. Thus, noting that $\innerp{S}{\phi}:=\lim_{\eps\to 0}\int_{\R^d}S_{\eps}(x)\psi(x)\dx $ and convergence in $L^p(\P)$ for $p>1$ implies convergence in probability, we aim to show that 
\begin{align}
\label{eq:L2 convergence in star scale invariant case}
\lim_{\eps\to 0}E_\eps=0.
\end{align}

Then let us define 
\begin{align*}
\Delta_{\gamma,\eps}(x,x'):&=\E\bigg[\bigg(\frac{1}{\gamma}\log(\int_{\R^d}\eta(y)e^{\gamma Y_{\eps,x}}\nu_\gamma^{\eps,x}(\dy ))
\\
&\qquad\quad-\E\abr{\frac{1}{\gamma}\log(\int_{\R^d}\eta(y)e^{\gamma Y_{\eps,x}}\nu_\gamma^{\eps,x}(\dy ))}\bigg)
\\
&\qquad\times\bigg(\frac{1}{\gamma}\log(\int_{\R^d}\eta(y)e^{\gamma Y_{\eps,x'}}\nu_\gamma^{\eps,x'}(\dy ))
\\
&\qquad\quad-\E\abr{\frac{1}{\gamma}\log(\int_{\R^d}\eta(y)e^{\gamma Y_{\eps,x'}}\nu_\gamma^{\eps,x'}(\dy ))}\bigg)\bigg]
\\
&=\Cov(A_\eps(x),A_\eps(x')),
\end{align*}
where
\[A_\eps(x):= \frac{1}{\gamma}\log(\int_{\R^d}\eta(u)e^{\gamma Y_{\eps,x}(u)}\nu_{\gamma}^{\eps,x}(\du ))\]
By Fubini, the limit in \Cref{eq:L2 convergence in star scale invariant case} is equivalent to 
\[\lim_{\eps\to 0}\int_{D\times D}\psi(x)\psi(x')\Delta_{\gamma,\eps}(x,x')\dx \dx '.\]
Therefore, since $\int_{D\times D}|\psi(x)\psi(x')|\dx \dx '<\infty$ we are done by dominated convergence, if we prove the lemma below. 

\begin{lemma}
For $\gamma\in [0,\sqrt{2d}]$ and for $x,x'\in \supp(\psi)$ such that $x\neq x'$ we have $\lim_{\eps\to 0}\Delta_{\gamma,\eps}(x,x')=0$. Furthermore, $\Delta_{\gamma,\eps}(x,x')$ is bounded uniformly in $x,x'\in\supp(\psi)$ and $\eps>0$.
\end{lemma}

\begin{proof}
We begin by considering the second claim. Firstly we have 
\begin{align*}
\sup_{x,x'\in\supp(\psi)}&\sup_{\eps>0}\Cov(A_\eps(x),A_\eps(x'))
\\
&=\sup_{x,x'\in\supp(\psi)}\sup_{\eps>0}\rbr{\E[A_\eps(x)A_\eps(x')]-\E[A_\eps(x)]\E[A_\eps(')]}
\\
&\leq\sup_{x,x'\in\supp(\psi)}\sup_{\eps>0}\rbr{\E[|A_\eps(x)A_\eps(x')|]+\E[|A_\eps(x)|]\E[|A_\eps(x')|]}
\\
&\leq 2\sup_{x,x'\in\supp(\psi)}\sup_{\eps>0}\sqrt{\E[|A_\eps(x)|^2]}\sqrt{\E[|A_\eps(x')|^2]}
\\
&\leq 2\rbr{\sup_{x\in\supp(\psi)}\sup_{\eps>0}\sqrt{\E[|A_\eps(x)|^2]}}^2.
\end{align*}
Thus, we need to show that
\[\sup_{x\in\supp(\psi)}\sup_{\eps>0}\E[|A_\eps(x)|^2]<\infty.\]
since the square root is irrelevant for the finiteness of the quantity.
With similar arguments as in the proofs of \Cref{lemma: even moments of Y} and \Cref{prop:finiteness of the counter term} we obtain
\begin{align*}
\E[|A_\eps(x)|^2]\leq C\rbr{\E[\sup_{|u|\leq 1}|Y_{\eps,x}(u)|^2]+\frac{1}{\gamma^2}\E\abr{|\log(\nu_\gamma^{\eps,x}(\eta))|^2}},
\end{align*}
where $C>0$ is some constant independent of $\eps$ and $x$. Now again the first term is bounded uniformly in $\eps$ and $x$ by \Cref{lemma: even moments of Y} and the second term is also uniformly bounded by the fact that $\nu_{\gamma}^{\eps,x}$ is distributed like $\nu_{\gamma,B_1(0)}$, \Cref{lemma:logarithmic estimate} and \Cref{thm: Existence and moments of star scale invariant GMC}. 

For the first claim, note first that \Cref{prop: convergence of Yx and Yx'} combined with the Continuous mapping theorem \cite[Theorem 4.27]{Ka02a} and the fact that $\nu_{\gamma}^{\eps,x}$ and $\nu_{\gamma}^{\eps,x'}$ are identically distributed and independent on $B(0,1)$ for small enough $\eps$ and fixed $x\neq x'$ we have 
\begin{align*}
(A_\eps(x),A_\eps(x'))\overset{d}{\rightarrow}
\bigg(\frac{1}{\gamma}\log(\int_{\R^d}\eta(u)e^{\gamma G_1(u)}\nu_{\gamma,1}(\du)),
\\
\qquad\frac{1}{\gamma}\log(\int_{\R^d}\eta(v)e^{\gamma G_2(v)}\nu_{\gamma,2}(\dv))\bigg),
\end{align*}
where $G_i$ are the independent Gaussian processes from \Cref{prop: convergence of Yx and Yx'} and $\nu_{\gamma,i}$ are distributed like $\nu_{\gamma,B_1(0)}$ and are independent of each other and of the processes $G_i$. Thus, the components above are independent and their covariance vanishes. By the estimates in the end of the proof, the covariance $\E[A_\eps(x)A_\eps(x')]$ converge to the covariance of the limit $\lim_{\eps\to 0} (A_\eps(x),A_\eps(x'))$. For an expert in the field the details of the above argument may be routine, but we give the full details for non-specialist reader interested in them.

Firstly, we note that for arbitrary real valued random variables we have $X_n\overset{d}{\to}X$ if $\varphi_{X_n}\to\varphi_X$, where $\phi_Z$ is the characteristic function of the random variable $Z$. Therefore, it is enough to show that 
\begin{align*}
\E\abr{e^{itA_\eps(x)+it'A_\eps(x')}}\to \E\bigg[\exp\bigg(\frac{it}{\gamma}\log(\int_{\R^d}\eta(u)e^{\gamma G_1(u)}\nu_{\gamma,1}(\du ))
\\
\qquad\qquad+\frac{it'}{\gamma}\log(\int_{\R^d}\eta(u)e^{\gamma G_2(u)}\nu_{\gamma,2}(\du ))\bigg)\bigg].
\end{align*}
We also note that by the fact that $\nu_\gamma^{\eps,x}$ is distributed like $\nu_{\gamma,B_{1}(0)}$ we have 
\begin{align*}
\E&\bigg[e^{itA_\eps(x)+it'A_\eps(x')}\bigg]
\\
&=\E\bigg[\exp\bigg(\frac{it}{\gamma}\log(\int_{\R^d}\eta(u)e^{\gamma Y_{\eps,x}(u)}\nu_{\gamma,1}(\du ))
\\
&\quad+\frac{it'}{\gamma}\log(\int_{\R^d}\eta(u)e^{\gamma Y_{\eps,x'}(u)}\nu_{\gamma,2}(\du ))\bigg)\bigg]
\\
&=\E\bigg[\E\bigg[\exp\bigg(\frac{it}{\gamma}\log(\int_{\R^d}\eta(u)e^{\gamma Y_{\eps,x}(u)}\nu_{\gamma,1}(\du ))
\\
&\quad+\frac{it'}{\gamma}\log(\int_{\R^d}\eta(u)e^{\gamma Y_{\eps,x'}(u)}\nu_{\gamma,2}(\du ))\bigg) \bigg|\sigma(\nu_{\gamma,1},\nu_{\gamma,2})\bigg]\bigg]
\\
&=\E_{\gamma}\bigg[\bigg(\E_{\eps,x,x'}\bigg[\exp\bigg(\frac{it}{\gamma}\log(\int_{\R^d}\eta(u)e^{\gamma Y_{\eps,x}(u)}\mu_1(\du ))
\\
&\quad+\frac{it'}{\gamma}\log(\int_{\R^d}\eta(u)e^{\gamma Y_{\eps,x'}(u)}\mu_2(\du ))\bigg)\bigg]\bigg)_{(\mu_1,\mu_2)=(\nu_{\gamma,1},\nu_{\gamma,2})}\bigg],
\end{align*}
where the last step holds by independence and the fact that the imaginary exponential is always uniformly bounded. Above $\E_\gamma$ and $\E_{\eps,x,x'}$ denote the expectations with respect to  $\nu_{\gamma,1}$ and $\nu_{\gamma,2}$ jointly, and $Y_{\eps,x}$ and $Y_{\eps,x'}$ jointly respectively. Now since everything above is uniformly bounded we can first exchange the order of the first expectations and the limit $\eps\to 0$. Then we can apply the strategy we started with backwards and use the Continuous mapping theorem to the inner expectation, where $\mu $ is a fixed realization of the multiplicative chaos $\nu_\gamma$. 

We know that $Y_{\eps,x}$ and $Y_{\eps,x'}$ converge jointly in distribution by  \Cref{prop: convergence of Yx and Yx'}. Thus, to show that for fixed positive Radon measures $\mu_1$ and $\mu_2$ we have
\begin{align*}
&\bigg(\frac{1}{\gamma}\log(\int_{\R^d}\eta(u)e^{\gamma Y_{\eps,x}(u)}\mu_1(\du)),\frac{1}{\gamma}\log(\int_{\R^d}\eta(u)e^{\gamma Y_{\eps,x'}(u)}\mu_2(\du))\bigg)
\\
&\overset{d}{\longrightarrow}\rbr{\frac{1}{\gamma}\log(\int_{\R^d}\eta(u)e^{\gamma G_1(u)}\mu_1(\du)),\frac{1}{\gamma}\log(\int_{\R^d}\eta(u)e^{\gamma G_2(u)}\mu_2(\du))},
\end{align*}
it suffices to show that the map $C(\overline{B(0,1)})\to \R$ 
\begin{align*}
f\mapsto\frac{1}{\gamma}\log(\int_{\R^d}\eta(u)e^{\gamma f(u)}\mu_i(\du ))
\end{align*}
is almost surely continuous at $G\in C(\overline{B_1(0)})$. Above $G$ is an independent copy of the components of the process $(G_1,G_2):=\lim_{\eps\to 0}(Y_{\eps,x},Y_{\eps,x'}) $, where the limit is in distribution. Thus, there are two steps: We must show
\begin{enumerate}
\item 
\begin{align*}
\P\rbr{\int_{\R^d}\eta(u)e^{\gamma Y_{\eps,z}(u)}\mu_i(\du )=0}=0 
\end{align*}
for $z=x,x'$, $i=1,2$ and $\forall \eps_0\geq\eps \text{ and }z\in\R^d$ with some fixed $\eps_0\in (0,1)$ 
\item
and that the map 
\begin{align*}
f\mapsto \int_{\R^d}\eta(u)e^{\gamma f(u)}\mu(\du )
\end{align*}
is almost surely continuous at $f=G$ for almost all realizations $\mu$ of $\nu_\gamma$. 
\end{enumerate}

For the first part we note that the exponential is always non-negative, and we have assumed that ${\int_{\R^d}\eta(x)\dx =1}$.  Therefore, by positivity we have
\begin{align*}
\E\abr{\int_{\R^d}\eta(u)e^{\gamma Y_{\eps,x}(u)}\nu_{\gamma,1}(\du )}=\int_{\R^d}\eta(u)\E[e^{\gamma Y_{\eps,x}(u)}\nu_{\gamma,1}(\du )]>0,
\end{align*}
because $Y_{\eps,x}$ is almost surely finite in the domain of $\nu$ and independent of $\nu_{\gamma,1}$, and the expectation of GMC measure is just the Lebesgue measure. Thus,  
\begin{align*}
\P\rbr{\int_{\R^d}\eta(u)e^{\gamma Y_{\eps_n,x}(u)}\mu_1(\du )>0}>0.
\end{align*}
Then for any sequence $\eps_n\to 0$ the upper bound $\eps_0$ translates to some fixed $\eps_N$ so that the event 
\begin{align*}
A:=\bigg\{ \int_{\R^d}\eta(u)e^{\gamma Y_{\eps_N,x}(u)}\mu_1(\du )=0 \biggm| n\geq N\bigg\}
\end{align*} 
is independent of the sigma-algebras $\sigma(Y_{\eps_n,x})$ for all $n<N$. Thus, the event belongs to the tail sigma-algebra and Kolmogorov's 0-1 law yields $\P(A)=0$, because we saw that $\P(A^c)>0$. We can do the same for the other case. 

Then let $f_n\to f$ in $C(\overline{B_1(0)})$. We have for any $\eps>0$ a large enough $n$ s.t. 
\begin{align*}
\bigg|\int_{\R^d}\eta(u)e^{\gamma f_n(u)}\mu(\du )-\int_{\R^d}&\eta(u)e^{\gamma f(u)}\mu(\du )\bigg|
\\
&\leq  \int_{\R^d}|\eta(u)|e^{\gamma f(u)}|1-e^{\gamma(f_n(u)-f(u))}|\mu(\du )\\
&\leq  \gamma\mu(|\eta|e^{\gamma f})\|f_n-f\|_{\infty}\exp(\gamma\|f_n-f\|_\infty)
\end{align*} 
where we have used $|1-e^x|<|x|e^{|x|}$ denoted $\|\cdot\|_\infty$ for the sup-norm in $C(\overline{B_1(0)})$. For any $\eps>0$ we can always choose $N_\eps$ so large that for $n\geq N_\eps$ the last expression is smaller than $\eps$. 

Therefore, we only need to justify the exchange of limits and integration in the Covariance. We have
\begin{align*}
\Cov(A_\eps(x),A_\eps(x'))=\E[A_\eps(x)A_\eps(x')]-\E[A_\eps(x)]\E[A_\eps(x')]
\end{align*}
We can justify the exchange of limits by uniform integrability of all relevant random variables above. This is easiest to do by the $L^p$-boundedness criterion. First we have 
\begin{align*}
\sup_{\eps>0}\E[(A_\eps(x)A_\eps(x'))^2]\leq \rbr{\sup_{\eps>0}\sqrt{\E[(A_\eps(x))^4]}}\rbr{\sup_{\eps>0}\sqrt{\E[(A_\eps(x'))^4]}},
\end{align*}
where we have used Cauchy-Schwarz. Thus, it is enough also for the individual expectations to bound the fourth moment of $A_\eps$ uniformly in $\eps>0$ and $x,x'$. By similar arguments as before we have
\begin{align*}
\E[(A_\eps(x))^4]\leq C\rbr{\E[\sup_{|u|\leq 1}|Y_\eps(u)|^4]+\E[|\log(\nu_\gamma(\eta))|^4]}<\infty 
\end{align*}
uniformly in $\eps>0$ by similar considerations as before.
\end{proof}

\subsection{Generalization to more general reference measures than the Lebesgue measure}
\label{sec:Generalization to more general reference measures than the Lebesgue measure}
We have restricted the discussion to the case, where the reference measure $\mu$ in the construction of the GMC measure is the Lebesgue measure. A natural question is: What parts of the argument would need to be modified if we wish to use more general reference measures? First of all, we might need work with subsequences even in the sub-critical case since the convergence may not be almost sure for general reference measures.

More importantly, the measures $\nu_{\gamma}^{\eps,x}$ in our fundamental representation \ref{eq: definition of nuGammaEpsx} would not be any more distributed like $\nu_{\gamma,B_1(0)}$, that is, the GMC measure constructed from a $\star$-scale invariant field $S$ on the unit ball $B_1(0)$. Indeed, at the level of approximations
\begin{align*}
\nu_{\gamma,\delta,\mu}^{\eps,x}(\du):=e^{\gamma Z_{(\eps,\delta,x)}(u)-\frac{\gamma^2}{2}\E[(Z_{(\eps,\delta,x)}(u))^2]}\eps^{-d}\mu_{\eps,x}(\du),
\end{align*}
where for any Borel set $A$, $\mu_{\eps,x}(A)=\mu(\phi_{\eps,x}^{-1}(A))$ is the image measure of $\mu$ under the map $\phi_{\eps,x}(y)= (x-y)/\eps$. Thus, $\nu_{\gamma,\mu}^{\eps,x}$ is no longer distributed like $\nu_{\gamma,B_1(0)}$. $\nu_{\gamma,\mu}^{\eps,x}$ would still be independent of $S_\eps$ and $Y_{\eps,x}$ and we also expect \Cref{prop: independence of nu eps x and nu eps x'} to hold with more general reference measures. 

The above fact that $\nu_{\gamma,\mu}^{\eps,x}$ is no longer distributed like $\nu_{\gamma,B_1{0}}$ means that we cannot any more automatically use \Cref{thm: Existence and moments of star scale invariant GMC} to estimate the moments in the proof of \ref{thm:field from the measure star scale invariant case}. Thus, we would need to estimate
\begin{align*}
\sup_{x\in \supp(\psi)}\sup_{\eps>0}\E\abr{\rbr{\int_{\R^d}\eta(u)\nu_{\gamma}^{\eps,x}(\du )}^a}
\end{align*}
for $a\in (-\delta,\delta)$ for some positive $\delta$. This would require some scaling and translation properties from the reference measure. 

Lastly, also the argument for the convergence in distribution of $(A_{\eps}(x),A_\eps(x'))$ used in the proof of \Cref{thm:field from the measure star scale invariant case} would need to be modified for a general reference measure $\mu$.

\section{Applications}
\label{sec:Applications}
In this section, we first give two examples of mildly non-Gaussian cases, where our results apply. For the second example we have that our results hold even with \Cref{def: non gaussian multiplicative chaos 2} for the non-Gaussian multiplicative chaos. Then we give two applications of our methods, both of which have the flavour that a small set of parameters determines the object of interest for all relevant values of these parameters. In these two latter applications, we only consider the Gaussian case. 
\subsection{Examples}
\label{sec:Examples}
In this section, we give two examples of mildly non-Gaussian log-correlated fields. The first one comes from quantum or statistical field theory. Namely, the massive sine-Gordon field $\Phi_{\mathrm{sG}}$. The law of this field is formally given by
\begin{align}
\label{eq:formal definition of sine-Gordon}
\nu_{\mathrm{sG}}(\d\Phi)=\frac{1}{Z}e^{-2z\int_{\Lambda}\cos(\sqrt{\beta}\Phi(x))dx}\nu_{\mathrm{GFF}(m)}(\d \Phi),
\end{align}
where $Z$ is a normalization constant,  $z\in \R$, $\beta\in[0,8\pi)$, $\Lambda\in\R^d$ a bounded domain and $\nu_{\mathrm{GFF}(m)}$ is the law of the massive Gaussian free field (GFF). The massive GFF is a log-correlated Gaussian field with a covariance kernel given by the Greens function of the massive Laplace operator $-\Lap+m^2$ on $\Lambda$ with appropriate boundary conditions. As mentioned already, the equality in \Cref{eq:formal definition of sine-Gordon} is only formal equivalence and the actual construction requires renormalization. The setup relevant to us is $\beta\in [0,6\pi)$, $\Lambda=\BB{T}^2$, that is, the two-dimensional unit torus. In this case, the sine-Gordon field has been constructed as a random element in $H^{-s}(\BB{T})$, $s>0$, and the difference of the GFF and the sine-Gordon field is a deterministically bounded Hölder-continuous field \cite[Theorem 1.2]{BaHo22a}. \Cref{thm:field from the measure for mildly non-Gaussian case} applies as is, but using \Cref{def: non gaussian multiplicative chaos 2} would require non-trivial estimates. It is easy to show that the Hölder part does have exponential moments, but we will not try to verify the assumptions in \Cref{subsec:Discussion about the other possible definition of non-Gaussian multiplicative chaos in theorem C} for this example. In the paper \cite{BaHo22a}, the field $\Phi_{\mathrm{sG}}$ is constructed as the final solution with $t=0$ to certain backwards SDE on $[0,\infty]$. The approximations to the field are then the $t>0$ solutions. To verify the assumptions for using \Cref{def: non gaussian multiplicative chaos 2}, one would probably need to use these approximations.

Our next example comes from the realm of probabilistic number theory. In studying the statistical behaviour of the Riemann zeta function $\zeta$ on the critical line $\Re(z)=\half$, the authors of \cite{SaWe20a} constructed a random Schwartz distribution called the randomized zeta function $\zeta_{\mathrm{rand}}(\half+ix)$, $x\in\R$, where $\zeta_{\mathrm{rand}}(s)$ to the right of the critical line $\Re(s)=\half$ is defined as the randomized Euler product 
\begin{align*}
\zeta_{\mathrm{rand}}(s):=\prod_{k=1}^\infty\rbr{\frac{1}{1-p_k^{-s}e^{2\pi i\theta_k}}}
\end{align*}
where $p_k$ is the $k^{\mathrm{th}}$ prime number and $(\theta_k)_{k\in\N}$ is an i.i.d sequence of uniform random variables on $[0,1]$. Note that for $\Re(s)=\half$ this expression is formal and indeed converges only as a Schwartz distribution. This distribution turns out to be related to complex GMC. In this context, the field that fits into our framework of mildly non-Gaussian log-correlated fields is formally $X(x):=\log(|\zeta_{\mathrm{rand}}(\half+ix)|)$ 
For the approximation $\zeta_{N,\mathrm{rand}}$, where one simply truncates the randomised Euler product at $N\in\N$ the authors show \cite[Theorem 1.8]{SaWe20a} that 
\begin{align}
\label{eq: multiplicative chaos measure and randomised zeta function}
\frac{|\zeta_{N,\mathrm{rand}}(\half+ix)|^{\gamma}}{\E\abr{|\zeta_{N,\mathrm{rand}}(\half+ix)|^{\gamma}}}dx:=\frac{e^{\gamma X_N(x)}}{\E\abr{e^{\gamma X_N(x)}}}dx,\quad x\in [0,1]
\end{align}
converges almost surely in the weak topology  for $\gamma<\sqrt{2}$ to a non-trivial random measure $\lambda_\gamma$ such that
\begin{align*}
\lambda_\gamma(\dx)=f_\gamma(x)\nu_\gamma(\dx),
\end{align*}
where $\nu_\gamma$ is a GMC measure and $f_\gamma$ is positive and a continuous random function such that $f_\gamma$ and $1/f_\gamma$ and their derivatives of any order are a.s. bounded. 

Furthermore, they show \cite[Theorem 1.7]{SaWe20a} that for any fixed $N\in\N$ and on any $I_A:=[-A,A]$, there exists a decomposition 
\begin{align*}
\log(\zeta_{N,\mathrm{rand}}(1/2+ix))=\Gcal_N(x)+\Ecal_N(x), \quad x\in I_A,
\end{align*}
where $\Gcal_N$ is a Gaussian process and $\Ecal_N$ is a smooth random function that converges a.s. to $\Ecal \in C^\infty (I_A)$. Furthermore, the following uniform exponential moments
\begin{align*}
\E\abr{\exp(\lambda\sup_{N\geq 1}\rbr{\sum_{k=0}^l\norm{\Ecal_N^{(k)}}_{L^\infty[0,1]}})}
\end{align*}
exist for all $\lambda>0$ and $l\geq  0$. Above $\Ecal_N^{(k)}$ is the $k^{\mathrm{th}}$ derivative of $\Ecal_N$ for $k\geq 1$ and the function itself for $k=0$. 

Since
\begin{align*}
X_n=\log(|\zeta_{N,\mathrm{rand}}(1/2+ix)|)=\Re(\Gcal_N(x))+\Re(\Ecal_N(x)),
\end{align*}
the relevant Gaussian field for us is $G:=\lim_{N\to\infty}\Re(\Gcal_N(x))$ on $[0,1]$. In \cite[Section 6]{SaWe20a}, the authors prove the convergence of the measures in \Cref{eq: multiplicative chaos measure and randomised zeta function} using Kahane's original framework, and along the way they show that the covariance kernels converge locally uniformly on $[0,1]\setminus \{x=y\}$ to a logarithmic kernel, that is,  
\begin{align*}
\Cov(\Re(\Gcal_N(x)),\Re(\Gcal_N(y)))\overset{n\to\infty}{\longrightarrow}\half\log(|x-y|^{-1})+g(x-y),
\end{align*}
where $g$ is even smooth on the enlarged domain $[-2,2]$. The process $\Gcal_N$ has an explicit expression given by 
\begin{align}
\Gcal_N(x):=\sum_{k=1}^N\frac{1}{\sqrt{2p_k}}(W_k^R+iW_k^I)p_k^{-ix}
\end{align}
where $\{W_k^R\}$ and $\{W_k^I\}$ are two sequences of independent standard Gaussian random variables that are also mutually independent.  

Then we need to check that the field $X:=\lim_{N\to\infty} X_N$ fulfils the conditions discussed in \Cref{subsec:Discussion about the other possible definition of non-Gaussian multiplicative chaos in theorem C}. Let us consider the function
\begin{align*}
g_N(x):=\frac{\E\abr{e^{\gamma \Gcal_N(x)}}}{\E\abr{e^{\gamma X_N(x)}}}=\prod_{k=1}^N\frac{\E\abr{e^{\frac{\gamma}{\sqrt{2p_k}}\Re([W_k^R+iW_k^I]p^{-ix})}}}{\E\abr{\abs{1-p_k^{-\half}p_k^{-ix}e^{i2\pi \theta_k}}^{-\gamma}}}:=\prod_{k=1}^N\frac{E_k^G(x)}{E_k^X(x)}.
\end{align*}
Firstly, both the nominator and denominator are independent of $x$, which can be seen from the fact that $\Gcal_N$ and $X_N$ are stationary processes, that is, they are invariant under translation or by doing an explicit computation. Thus, it is enough to consider $E_k^G\equiv E_k^G(0)$ and $E_k^X\equiv E_k^X(0)$. By a standard computation, we have
\begin{align*}
E_k^G=e^{\frac{\gamma^2}{4p_k}}.
\end{align*}
Next consider the auxiliary function
\begin{align*}
f(\lambda):=\log(\E\abr{|1-\lambda e^{i2\pi \theta}|})=\log(\int_0^{2\pi}|1-\lambda e^{is}|^{-\gamma}\frac{ds}{2\pi})
\end{align*}
where $\theta$ is uniformly distributed on $[0,1]$ and $\lambda\in \mathbb{D}\subset \C$ with $\mathbb{D}$ denoting the open unit disc. Then we can write 
\begin{align*}
\int_0^{2\pi}|1-\lambda e^{is}|^{\gamma}\frac{ds}{2\pi}&=\int_0^{2\pi}(1-\frac{\gamma}{2} \lambda e^{is}+\bigO_\gamma(\lambda^2e^{i2s}))(1-\frac{\gamma}{2}\bar{\gamma}e^{-is}+\bigO_\gamma(\bar{\lambda}^2e^{-i2s})\frac{ds}{2\pi}
\\
&=1+\frac{\gamma^2}{4}|\lambda|^2+\bigO_\gamma(|\lambda|^4)
\end{align*} 
by Fourier-orthogonality. Above the notation $\bigO_\gamma$ means that the implied constant depends on $\gamma$. Therefore,
\begin{align*}
f(\lambda)=\frac{\gamma^2}{4}|\lambda|^2+\bigO_\gamma(|\lambda|^4)
\end{align*}
and we obtain
\begin{align*}
g_N(x)=\prod_{k=1}^Ne^{\frac{\gamma^2}{4p_k}}e^{-f(p_k^{-\half})}=e^{-\sum_{k=1}^N\bigO_\gamma(p_k^{-2})}=\exp(-\bigO_\gamma\rbr{\sum_{k=1}^Np_k^{-2}}).
\end{align*}
Since the series in the exponent converges, the limit $g(x):=\lim_{n\to\infty}g_N(x)$ is a non-zero constant.
Thus, we have a non-trivial example where our results will apply even with \Cref{def: non gaussian multiplicative chaos 2}.
\subsection{Thick points determine the field everywhere}
\label{sec:Thick points determine the field everywhere}
It is standard fact that if $f\colon \R^d\supset D\to\R$ is a continuous function and $f=g$ on a dense subset $D'\subset D$, then $f=g$ on the whole domain $D$. However, the fields we are considering are not continuous functions, but Schwartz distributions so it would be surprising if something like the above would still hold. It turns out that there is enough structure in the set of thick points of the field that we can formulate something similar to the above statement.  

In this section, we will be working in the sub-critical regime $\gamma<\sqrt{2d}$. A carrier of the GMC measure $\mu_{\gamma,G}$ is contained in the set of thick points
\begin{align*}
\Tcal_{\gamma,G}:=\bigg\{x\in D\,\,\bigg|\,\, \lim_{\eps\to 0}\frac{G_\eps(x)}{\log(\eps^{-1})}=\gamma\bigg\}
\end{align*}
and furthermore this set has Hausdorff dimension less than $d-\gamma^2/2$ \cite{Ka85a}. 
Thus, the set $\Tcal_{\gamma,G}$ is a fractal set in $D$. It is not obvious that it is even dense. Therefore, there should not be any reason to a priori expect that the values of the field $G$ on $\Tcal_{\gamma,G}$ would completely determine it. Using the results in \cite{JuLaWe24a} we will formulate a slightly weaker statement, which is still non-trivial and non-intuitive. Namely, that the field $G$ is measurable with respect to the sigma-algebra generated by the uniform measures on the set of thick points at each spatial scale. To be more precise, let $\{\eps_n\}_{n\in\N}$ with $\eps_n<1$ for all $n\in\N$ and $\eps_n\to 0$ as $n \to \infty$ be an arbitrary deterministic sequence, and denote
\begin{align*}
\Tcal_{\gamma,G}^{(n)}:=\{x\in D \mid G_{\eps_n}(x)\geq \gamma \log(\eps_n^{-1})\}
\end{align*}
and 
\begin{align*}
\rho_{\gamma,G,n}(dx):=\frac{\1\{\Tcal_{\gamma,G}^{(n)}\}}{\P(\Tcal_{\gamma,G}^{(n)})}dx.
\end{align*}
Then, $\rho_{\gamma,G,n}\to \mu_{\gamma,G}$ in the vague topology of convergence of measures and in probability \cite[Corollary 1.9]{JuLaWe24a}\footnote{Recall that if $\mu_n(f)\to\mu(f)$ for every $f\in C_c(D)$, then we say that $\mu_n$ converges to $\mu$ in the vague topology of convergence of measures on $D$}. Similar earlier results also exist in the special cases of discrete GFF \cite{BiLo19a} and the logarithm of the zeta function \cite{ArHaKi22a}. Then since there exists a subsequence such that the convergence is almost sure, there exists a version of $\nu_{\gamma,G}$, that is measurable with respect to
\begin{align*}
\Sigma_{\gamma,G}:=\bigcap_{N_0=1}^\infty\sigma(\rho_{\gamma,G,n}, \,\, n\geq N_0):=\bigcap_{N_0=1}^\infty\sigma(\rho_{\gamma,G,n}(f), \,\, f\in C_c(D),\,\, n\geq N_0).
\end{align*}  
Similarly, our reconstruction yields that for any $\gamma$ there exists a version the field $G$ that is measurable with respect to
\begin{align*}
\sigma(\mu_{\gamma,G}):=\sigma(\mu_{\gamma,G}(f),\,\, f\in C_c).
\end{align*}
Thus, since a composite function of two measurable functions is again measurable we get the following result.
\begin{theorem}
\label{thm:Measurability of G with respect to the thick points}
For any $\gamma\in[0,\sqrt{2d})$, the log-correlated Gaussian field $G$ is measurable with respect to $\Sigma_{\gamma,G}$. 
\end{theorem}

\subsection{Knowing the GMC for one value of $\gamma$ determines it for all the other appropriate values of $\gamma$}
\label{sec:Knowing the GMC for one value of gamma determines it for all the other appropriate values of gamma}
In this section, we will only consider the $\star$-scale invariant fields in the $L^2$-phase of the sub-critical regime, that is, $\gamma\in [0,\sqrt{d})$. However, we expect that the result of this section can be generalized. 

On an abstract level with a similar measurability argument as in the previous section, we have that the GMC measure $\nu_{\gamma_0}$ with fixed $\gamma_0$ determines the GMC measures $\nu_\gamma$ for all other appropriate values of $\gamma$, where $\nu_\gamma\equiv \nu_{gamma,S}$. This is, because the underlying field is determined by the field independently of the value of $\gamma_0$ used in the reconstruction and the field in turn determines the GMC measure for all values of $\gamma$.  As an application of our methods, we show that this can be demonstrated more concretely. That is, we show that the construction of the GMC measure can be done using the approximation of the field used in our reconstruction result. Let us define for fixed $\gamma_0\in [0,\sqrt{d})$
\begin{align*}
S_{\eps,\gamma_0}(x):=\frac{1}{\gamma_0}\log(\int_{\R^d}\eta_\eps(y-x)\nu_{\gamma_0,S}(\dy))-F_{\gamma_0,\eps,\eta}(x).
\end{align*}
for $x\in K\subset D$ with $K$ a compact set such that 
$\dist(K,\partial D)\leq \eps$ for all $0<\eps<\eps_0 \leq 1$ with some fixed $\eps_0$.
We will show that for any other $\gamma\in [0,\sqrt{d})$ the limit
\begin{align*}
\lim_{\eps\to 0}\nu_{\gamma,\gamma_0,\eps}(\dx):=\lim_{\eps\to 0}\frac{e^{\gamma S_{\eps,\gamma_0}(x)}}{\E[e^{\gamma S_{\eps,\gamma_0}(x)}]}\dx
\end{align*} 
exists in probability with respect to the weak topology on $K$, is independent of $\gamma_0$ and equals the multiplicative chaos measure $\nu_{\gamma}$. To this end, we will show the convergence in probability for a fixed test function $f\in C_(K)$, which then implies the convergence in the weak topology by a standard argument. 

\begin{theorem}
\label{thm: one gamma determines GMC for all gamma}
Let $f\in C(K)$ and $K\subset D$ as described above. Also let $\gamma,\gamma_0\in [0,\sqrt{d})$ with $\gamma_0$ fixed. Then
\begin{align*}
\lim_{\eps\to 0}\nu_{\gamma,\gamma_0,\eps}(f)\overset{\P}{\longrightarrow}\nu_{\gamma}(f).
\end{align*}
\end{theorem} 

\begin{remark}
Since the counter term $F_{\gamma_0,\eps,\eta}$ in $S_{\eps,\gamma_0}$ is deterministic, it will cancel out in $\nu_{\gamma,\gamma_0,\eps}$ even before we take the above limit. Thus, we can work with 
\begin{align*}
\tilde{S}_{\eps,\gamma_0}(x):=\frac{1}{\gamma_0}\log(\int_{\R^d}\eta_\eps(y-x)\nu_{\gamma_0,S}(\dy)).
\end{align*} 
Henceforth, we will drop the tilde from the notation.
\end{remark}
Note that we can write
\begin{align*}
\frac{e^{\gamma S_{\eps,\gamma_0}(x)}}{\E[e^{\gamma S_{\eps,\gamma_0}(x)}]}=\frac{\rbr{\int_{\R^d}\eta_\eps(x-y)\nu_{\gamma_0,S}(\dy)}^{\frac{\gamma}{\gamma_0}}}{\E\abr{\rbr{\int_{D}\eta_\eps(x-y)\nu_{\gamma_0,S}(\dy)}^{\frac{\gamma}{\gamma_0}}}}.
\end{align*}
Recall that the threshold for the existence of the moments is $q_c=2d/\gamma^2$ so that we need to have 
\[2d/\gamma_0^2-\gamma/\gamma_0=\gamma_0^{-1}\sqrt{2d}(\sqrt{2d}/\gamma_0-\gamma/\sqrt{2d})>0,\] which is always satisfied for $\gamma,\gamma_0\in [0,\sqrt{d})$. Thus, the normalizing factor in the denominator is well-defined.
\begin{proof}[Proof of \Cref{thm: one gamma determines GMC for all gamma}]

\noindent
We will show that $\lim_{\eps\to 0}E_{\gamma,\gamma_0,\eps}=0$, where
\begin{align*}
E_{\gamma,\gamma_0,\eps}:=\E\abr{\rbr{\int_{D}f(x)\frac{e^{\gamma S_{\eps,\gamma_0}(x)}}{\E[e^{\gamma S_{\eps,\gamma_0}(x)}]}\dx-\int_{D}f(x)e^{\gamma S_{\eps}(x)-\frac{\gamma^2}{2}\E[(S_{\eps}(x))^2]}\dx}^2}.
\end{align*} 
First recall the fundamental representation 
\begin{align*}
\int_{\R^d}\eta_\eps(x-y)\nu_{\gamma_0,S}(\dy)=e^{\gamma_0 S_\eps(x)-\frac{\gamma_0^2}{2}\E[(S_\eps(x))^2]}\int_{\R^d}\eta(u)e^{\gamma_0 Y_{\eps,x}(u)}\nu_{\gamma_0}^{\eps,x}(\du).
\end{align*}
given in \Cref{eq: definition of nuGammaEpsx}.
Thus, we can rewrite 
\begin{align*}
\frac{e^{\gamma S_{\eps,\gamma_0}(x)}}{\E[e^{\gamma S_{\eps,\gamma_0}(x)}]}&=\frac{e^{\gamma S_\eps(x)}\rbr{\int_{\R^d}\eta(u)e^{\gamma_0 Y_{\eps,x}(u)}\nu_{\gamma_0}^{\eps,x}(\du)}^\frac{\gamma}{\gamma_0}}{\E\abr{e^{\gamma S_\eps(x)}\rbr{\int_{\R^d}\eta(u)e^{\gamma_0 Y_{\eps,x}(u)}\nu_{\gamma_0}^{\eps,x}(\du)}^\frac{\gamma}{\gamma_0}}}
\\
&=\frac{e^{\gamma (S_\eps(x)+A_{\eps}(x))}}{\E\abr{e^{\gamma (S_\eps(x)+A_{\eps}(x))}}}.
\end{align*}
where we have used the notation of \Cref{sec: Proof of the first main theorem} to condense the notation. That is,
\begin{align*}
A_\eps(x)=\frac{1}{\gamma_0}\log(\int_{\R^d}\eta(u)e^{\gamma_0 Y_{\eps,x}(u)}\nu_{\gamma_0}^{\eps,x}(\du))\equiv A_{\gamma_0,\eps}(x).
\end{align*}
Above on the RHS we have added the subscript $\gamma_0$ to highlight the importance of the dependence of $A_{\eps}$ on $\gamma_0$, and we will continue to do so in this section. We can also rewrite 
\begin{align*}
E_{\gamma,\gamma_0,\eps}=\E\abr{\rbr{\int_{\R^d}f(x)\frac{e^{\gamma S_\eps(x)}}{\E\abr{e^{\gamma S_\eps(x)}}}\cbr{\frac{\E\abr{e^{\gamma S_\eps(x)}}e^{\gamma A_{\gamma_0,\eps}(x)}}{\E\abr{e^{\gamma (S_\eps(x)+A_{\gamma_0,\eps}(x))}}}-1}\dx}^2}.
\end{align*}
Writing out the square yields
\begin{align}
\int_{D\times D}&\bigg(f(x)f(y)\frac{
e^{\gamma (S_\eps(x)+S_\eps(y))}}{\E\abr{e^{\gamma S_\eps(x)}}\E\abr{e^{\gamma S_\eps(y)}}}\nonumber
\\
&\qquad\qquad\times\prod_{u\in\{x,y\}}\cbr{\frac{\E\abr{e^{\gamma S_\eps(u)}}e^{\gamma A_{\gamma_0,\eps}(u)}}{\E\abr{e^{\gamma (S_\eps(u)+A_{\gamma_0,\eps}(u)})}}-1}\bigg)dxdy
\nonumber
\\
=&\int_{D\times D}f(x)f(y)\bigg(\frac{e^{\gamma (S_\eps(x)+S_\eps(y)+A_{\gamma_0,\eps}(x)+A_{\gamma_0,\eps}(y))}}{\E\abr{e^{\gamma (S_\eps(x)+A_{\gamma_0,\eps}(x))}}\E\abr{e^{\gamma (S_\eps(y)+A_{\gamma_0,\eps}(y))}}} \nonumber
\\
&-\frac{e^{\gamma (S_\eps(x)+S_\eps(y)+A_{\gamma_0,\eps}(x))}}{\E\abr{e^{\gamma S_\eps(y)}}\E\abr{e^{\gamma (S_\eps(x)+A_{\gamma_0,\eps}(x))}}}-\frac{e^{\gamma (S_\eps(x)+S_\eps(y)+A_{\gamma_0,\eps}(y))}}{\E\abr{e^{\gamma S_\eps(x)}}\E\abr{e^{\gamma (S_\eps(y)+A_{\gamma_0,\eps}(y))}}} \nonumber
\\
\label{eq: gamma_0 difference of squares}
&+\frac{
e^{\gamma (S_\eps(x)+S_\eps(y))}}{\E\abr{e^{\gamma S_\eps(x)}}\E\abr{e^{\gamma S_\eps(y)}}}\bigg)dxdy.
\end{align}

We can use Fubini and dominated convergence, if we prove upper bounds term-wise for each four terms, when we break up the above integral. This is because each term is positive. The last term is then the one that sets the benchmark towards which we expect each term to converge. We have
\begin{align}
\label{eq:upper bound for the covariance exponential}
\frac{\E\abr{e^{\gamma(S_\eps(x)+S_\eps(y))}}}{\E\abr{e^{\gamma S_\eps(x)}}\E\abr{e^{\gamma S_\eps(y)}}}=e^{\frac{\gamma^2}{2}\E[S_\eps(x)S_{\eps}(y)]}= \bigO(|x-y|^{-\frac{\gamma^2}{2}}),
\end{align}
where the implied constant is independent of $\eps$ and the last expression is integrable on $D\times D$ for $\gamma\in [0,\sqrt{d})$. 

Thus, we need to prove the following lemma
\begin{lemma}
We have the following decompositions.
\begin{align*}
\frac{\E[e^{\gamma (S_\eps(x)+S_\eps(y)+A_{\gamma_0,\eps}(x)+A_{\gamma_0,\eps}(y))}]}{\E\abr{e^{\gamma (S_\eps(x)+A_{\gamma_0,\eps}(x))}}\E\abr{e^{\gamma (S_\eps(y)+A_{\gamma_0,\eps}(x))}}}:&=e^{\frac{\gamma^2}{2}\E[S_\eps(x)S_{\eps}(y)]}F_{\eps}^{(1)}(x,y)
\\
\frac{\E[e^{\gamma (S_\eps(x)+S_\eps(y)+A_{\gamma_0,\eps}(x))}]}{\E\abr{e^{\gamma S_\eps(y)}}\E\abr{e^{\gamma (S_\eps(x)+A_{\gamma_0,\eps}(x))}}}:&=e^{\frac{\gamma^2}{2}\E[S_\eps(x)S_{\eps}(y)]}F_{\eps}^{(2)}(x,y)
\\
\frac{\E[e^{\gamma (S_\eps(x)+S_\eps(y)+A_{\gamma_0,\eps}(y))}]}{\E\abr{e^{\gamma S_\eps(x)}}\E\abr{e^{\gamma (S_\eps(y)+A_{\gamma_0,\eps}(y))}}}:&=e^{\frac{\gamma^2}{2}\E[S_\eps(x)S_{\eps}(y)]}F_{\eps}^{(3)}(x,y)
\\
&=e^{\frac{\gamma^2}{2}\E[S_\eps(x)S_{\eps}(y)]}F_{\eps}^{(2)}(y,x),
\end{align*}
where $F_\eps^{(i)}(x,y)\to 1$ as $\eps\to 0$ for $i=1,2,3$ and for all $x,y\in D$. Furthermore, there exists $g_i(x,y)$ such that $|F_\eps^{i}(x,y)|\leq g_i(x,y)$ uniformly in $\eps\in (0,1)$ and $g_i|\cdot-\cdot|^{-\gamma^2/2}\in L^1(D\times D)$
\end{lemma} 

\begin{proof}
We can do all the rest of the involved expectations at once by considering
\begin{align*}
\E[e^{\gamma[\sigma_x S_\eps(x)+\sigma_yS_{\eps}(y)+a_xA_{\gamma_0,\eps}(x)+a_yA_{\gamma_0,\eps}(y)]}],
\end{align*}
where $\sigma_z,a_z\in \{0,1\}$ for $z=x,y$. 

Next let us use the Girsanov-Cameron-Martin change of measure. We will need it in the form
\begin{align*}
\E[e^{X_1+X_2}F(Y_1,Y_2)]&=e^{\half\E[X_1^2]+\half\E[X_2^2]+\E[X_1X_2]}
\\
&\quad\times\E\abr{F\rbr{Y_1+\sum_{i=1}^2\E[Y_1X_i],Y_2+\sum_{i=1}^2\E[Y_2X_i]}},
\end{align*}
where $Y_i$, $i=1,2$ are Gaussian fields on $D\subset\R^d$ with realizations on suitable space $\Fcal$ and $X_i$, $i=1,2$ are Gaussian random variables, that are also jointly Gaussian with the fields $Y_i$. For us $X_i=\gamma\sigma_{z}S_{\eps}(z)$ and $Y_i=Y_{\eps,z}$, where $z=x$ for $i=1$ and $z=y$ for $i=2$. \cite[Theorem 2.8]{Da06a} covers this case. Then we have by independence and the uniform bounds proven below
\begin{align*}
\E[&e^{\gamma[\sigma_x S_\eps(x)+\sigma_yS_{\eps}(y)+a_xA_{\gamma_0,\eps}(x)+a_yA_{\gamma_0,\eps}(y)]}]
\\
=&\E\bigg[\E\bigg[e^{\gamma_0(\sigma_xS_\eps(x)+\sigma_yS_\eps(y))}\rbr{\int_{\R^d}\eta(u)e^{\gamma_0Y_{\eps,x}(u)}\mu_1(\du)}^{a_x\frac{\gamma}{\gamma_0}}
\\
&\quad\times\rbr{\int_{\R^d}\eta(u)e^{\gamma_0Y_{\eps,y}(u)}\mu_2(\du)}^{a_y\frac{\gamma}{\gamma_0}}\bigg]\bigg\vert_{\nu_{\gamma_0}^{\eps,z}=\mu_i}\bigg]
\\
=&\E\bigg[\E\bigg[e^{\gamma_0(\half\sigma_x\gamma\E[(S_\eps(x))^2]+\half\sigma_y\gamma\E[(S_\eps(y))^2]+\sigma_x\sigma_y\gamma\E[S_\eps(x)S_\eps(y)])}
\\
&\quad\times\rbr{\int_{\R^d}\eta(u)e^{\gamma_0(Y_{\eps,x}(u)+\sum_{z=x,y}\sigma_z\gamma\E[(S_\eps(z)Y_{\eps,x}(u)])}\mu_1(\du)}^{a_x\frac{\gamma}{\gamma_0}}
\\
&\quad\times\rbr{\int_{\R^d}\eta(u)e^{\gamma_0(Y_{\eps,y}(u)+\sum_{z=x,y}\sigma_z\gamma\E[(S_\eps(z)Y_{\eps,y}(u)])}\mu_2(\du)}^{a_y\frac{\gamma}{\gamma_0}}
\bigg]\bigg\vert_{\nu_{\gamma_0}^{\eps,z}=\mu_i}\bigg]
\\
=&\E\bigg[ e^{\gamma_0(\half\sigma_x\gamma\E[(S_\eps(x))^2]+\half\sigma_y\gamma\E[(S_\eps(y))^2]+\sigma_x\sigma_y\gamma\E[S_\eps(x)S_\eps(y)])}
\\
&\quad\times\rbr{\int_{\R^d}\eta(u)e^{\gamma_0(Y_{\eps,x}(u)+\sum_{z=x,y}\sigma_z\gamma\E[(S_\eps(z)Y_{\eps,x}(u)])}\nu_{\gamma_0}^{\eps,x}(\du)}^{a_x\frac{\gamma}{\gamma_0}}
\\
&\quad\times\rbr{\int_{\R^d}\eta(u)e^{\gamma_0(Y_{\eps,y}(u)+\sum_{z=x,y}\sigma_z\gamma\E[(S_\eps(z)Y_{\eps,y}(u)])}\nu_{\gamma_0}^{\eps,y}(\du)}^{a_y\frac{\gamma}{\gamma_0}}
\bigg],
\end{align*}
where the outer expectation is with respect to  $\nu_{\gamma_0}^{\eps,x}$ and $\nu_{\gamma_0}^{\eps,y}$
jointly and the inner expectations is with respect $S_\eps(x),S_\eps(y), (Y_{\eps,x}(u))_{u\in B_1(0)}$ and $(Y_{\eps,x}(u))_{u\in B_1(0)}$ jointly with fixed $x,y\in \supp(f)$. In the fixing of the realization in the middle steps $z=x$ if $i=1$ and $z=y$ if $i=2$. We have used the Cameron-Martin-Girsanov to the inner expectation. 
  
Thus, we readily see that in all three terms in \Cref{eq: gamma_0 difference of squares} the variance terms also cancel, and we obtain the same covariance term as in the last term. Then we just need to show that the rest of each term converges to $1$ for each term separately. To simplify notation let us write
\begin{align*}
B_{\eps,z}^{\gamma,\gamma_0}(x,y):=\rbr{\int_{\R^d}\eta(u)e^{\gamma_0Y_{\eps,z}(u)}e^{\gamma\gamma_0L_{\eps,z}^{\gamma,\gamma_0}(x,y,u)}\nu_{\gamma_0}^{\eps,z}(\du)}^{a_z\frac{\gamma}{\gamma_0}},
\end{align*}
where
\begin{align*}
L_{\eps,z}^{\gamma,\gamma_0}(u,x,y)=\sigma_x\E[Y_{\eps,z}(u)S_{\eps}(x)]+\sigma_y\E[Y_{\eps,z}(u)S_{\eps}(y)].
\end{align*}
Thus, we need to prove the following 
\begin{align*}
\frac{\E[B_{\eps,x}^{\gamma,\gamma_0}(x,y)B_{\eps,y}^{\gamma,\gamma_0}(x,y)]}{\E[B_{\eps,x}^{\gamma,\gamma_0}(x,y)]\E[B_{\eps,x}^{\gamma,\gamma_0}(x,y)]}\overset{\eps\to 0}{\longrightarrow} 1,
\end{align*}
which will prove the conditions $F_\eps^{(i)}(x,y)\to 1$. To this end, we show that $\{B_{\eps,z}^{\gamma,\gamma_0}(x,y)\}_{\eps\in (0,1)}$ for $z=x,y$ and $\{B_{\eps,x}^{\gamma,\gamma_0}(x,y)B_{\eps,y}^{\gamma,\gamma_0}(x,y)\}_{\eps\in (0,1)}$ are uniformly integrable with strictly positive expectations so that they converge in $L^1$. In addition, we show that the limits $\lim_{\eps\to 0}B_{\eps,z}^{\gamma_0,\gamma}(x,y)$ (where the limit is in distribution) for $z=x$ and $z=y$ are identically distributed and independent. Note that our estimates below also give the desired upper bounds. For the first claim, we show that
\begin{align}
\label{eq: UI of B}
\sup_{\eps\in (0,1)}\E[|B_{\eps,x}^{\gamma,\gamma_0}(x,y)B_{\eps,y}^{\gamma,\gamma_0}(x,y)|^{1+\delta}]<\infty
\end{align} 
for some $\delta>0$. Note that we also have
\begin{align*}
\label{eq:Boundedness from belov of the exponential moments of B}
\E[B_{\eps,z}^{\gamma,\gamma_0}]\geq & \E\abr{\inf_{u\in B_1(0)}e^{\gamma_0 Y_{\eps,z}(u)}}\inf_{\substack{u\in B_1(0)\\ x,y\in\supp(f)}}e^{\gamma\gamma_0 L_{\eps,z}^{\gamma,\gamma_0}(x,y,u)}
\\
\geq& \E\abr{e^{-\gamma \sup_{u\in B_1(0)}|Y_{\eps,z}(u)|}}\exp(-\gamma\gamma_0\sup_{\substack{u\in B_1(0) \\ x,y\in\supp(f)}}|L_{\eps,z}^{\gamma,\gamma_0}(x,y,u)|),
\end{align*}
where we used the facts that $Y_{\eps,z}$ is independent of $\nu_{\gamma_0}^{\eps,z}$, and the expectation of the latter is just the Lebesgue measure and $\int \eta dx=1$. The first factor above is strictly positive by \Cref{lemma: even moments of Y} and below we show that the supremum in the exponent of the second factor is finite. For the second claim, we show convergence in distribution with the same strategy that we used to prove that $(A_{\gamma_0,\eps}(x),A_{\gamma_0,\eps}(y))$ converge in distribution to two copies of identically distributed random variables in the proof of \Cref{thm:field from the measure star scale invariant case}.

Let us consider the second claim first. Now instead of proving one function is continuous, we have to show that a sequence of functions corresponding to arbitrary deterministic sequence $\eps_n\to 0$ satisfies the conditions of the continuous function theorem \cite[Theorem 4.27]{Ka02a}. The sequence of functions $F_n\colon C(\overline{B_1(0)})\to\R$ is given by 
\begin{align*}
f\mapsto \int_{\R^d}\eta(u)e^{\gamma_0 f(u)}e^{\gamma\gamma_0 L_{\eps_n,x}^{\gamma,\gamma_0}(x,y,u)}\mu(\du),
\end{align*}
where $\mu$ is an arbitrary realization of $\nu_{\gamma,B_1(0)}$. The other case is similar. Thus, we need to show that $F_n(f_n)\to F(f)$ for arbitrary $f\in C(\overline{B_1(0)})$ and any sequence $f_n\to f$ in $C(\overline{B_1(0)})$. Above $F$ is such that $F_n(f)\to F(f)$ in $\R$ for fixed $f\in C(\overline{B_1(0)})$. In the proof that $(A_{\eps,\gamma_0}(x),A_{\eps,\gamma_0}(y))$ converges in distribution we showed the continuity of the above map without the factor $\exp(\gamma\gamma_0L_{\eps_n,x}^{\gamma,\gamma_0}(x,y,u))$. Thus, we are done if we have a uniform upper bound for this term. First we have
\begin{align*}
|\E[Y_{\eps_n,x}(u)S_{\eps_n}(z)]|\leq&\int_1^{\eps_n^{-1}}|k_1(s|x-z+\eps_n u|)-k_1(s|x-z|)|\frac{\ds}{s}
\\
\leq &\int_{0}^1|k_1(s|\eps_n^{-1}(x-z)+u|)-k_1(s|\eps_n^{-1}(x-y)|)|\frac{\ds}{s}
\\
\leq &\frac{|u|^\alpha}{\alpha},
\end{align*}
where $\alpha>0$ is the Hölder-exponent of $k_1$. Furthermore,
\begin{align*}
\E[Y_{\eps_n,x}(u)S_{\eps_n}(z)]\overset{n\to\infty}{\longrightarrow}
\begin{cases}
l(u):=\int_0^1(k_1(s|u|)-1)\frac{\ds}{s} &\text{ if } z=x
\\
0 &\text{ if } z=y.
\end{cases}
\end{align*}
Therefore, we have
\begin{align*}
e^{\gamma\gamma_0L_{\eps_n,z}^{\gamma,\gamma_0}(x,y,u)}\leq  \sup_{u\in B_1(0)}e^{(\sigma_x+\sigma_y)\frac{\gamma\gamma_0}{\alpha}|u|^\alpha}\leq C,
\end{align*}
where the constant $C>0$ is independent of $x,y$ and $n$. Therefore, by the argument given in the proof of \ref{thm:field from the measure star scale invariant case} for the convergence in distribution of $(A_\eps(x),A_\eps(x'))$ we have
\begin{align*}
\bigg(\int_{\R^d}&\eta(u)e^{\gamma_0 Y_{\eps,x}(u)}e^{\gamma\gamma_0 L_{\eps_n,x}^{\gamma,\gamma_0}(x,y,u)}\nu_{\gamma_0}^{\eps,x}(\du),
\\
&\int_{\R^d}\eta(u)e^{\gamma_0 Y_{\eps,y}(u)}e^{\gamma\gamma_0 L_{\eps_n,x}^{\gamma,\gamma_0}(x,y,u)}\nu_{\gamma_0}^{\eps,y}(\du)\bigg)
\\
&\overset{d}{\longrightarrow}\bigg(\int_{\R^d}\eta(u)e^{\gamma_0 G_1(u)}e^{\gamma\gamma_0\sigma_x l(u)}\nu_{\gamma_0,1}(\du),
\\
&\qquad\qquad\int_{\R^d}\eta(u)e^{\gamma_0 G_2(u)}e^{\gamma\gamma_0\sigma_y l(u)}\nu_{\gamma_0,2}(\du)\bigg),
\end{align*}
where $\nu_{\gamma_0, i}$, $i=1,2$ are two independent copies of $\nu_{\gamma_0,B_1(0)}$ and $(G_1,G_2)$ is the limit in distribution of $(Y_{\eps,x},Y_{\eps,y})$, where the components are independent and identically distributed as shown in \Cref{prop: convergence of Yx and Yx'}. 

Then let us prove \Cref{eq: UI of B}. We have by Hölder's  inequality 
\begin{align*}
\E&[|B_{\eps,x}^{\gamma,\gamma_0}(x,y)B_{\eps,y}^{\gamma,\gamma_0}(x,y)|^{1+\delta}]
\\
&\leq C\E\bigg[e^{\gamma(1+\delta) \sup_{u\in B_1(0)}|Y_{\eps,x}|}e^{\gamma(1+\delta) \sup_{u\in B_1(0)}|Y_{\eps,x}|}
\\
&\qquad\qquad\times(\nu_{\gamma}^{\eps,x}(|\eta|))^{\sigma_x\frac{\gamma}{\gamma_0}(1+\delta)}(\nu_{\gamma}^{\eps,y}(|\eta|))^{\sigma_y\frac{\gamma}{\gamma_0}(1+\delta)}\bigg]
\\
&\leq 
C\E\abr{\rbr{e^{\gamma(1+\delta) \sup_{u\in B_1(0)}|Y_{\eps,x}|}}^{p}}^\frac{1}{p}\E\abr{\rbr{e^{\gamma(1+\delta) \sup_{u\in B_1(0)}|Y_{\eps,y}|}}^{p}}^\frac{1}{p}
\\
&\quad\times\E\abr{\rbr{\nu_{\gamma_0,1}(|\eta|)^{\sigma_x\frac{\gamma}{\gamma_0}(1+\delta)}}^q}^\frac{1}{q}\E\abr{\rbr{\nu_{\gamma_0,1}(|\eta|)^{\sigma_y\frac{\gamma}{\gamma_0}(1+\delta)}}^q}^\frac{1}{q}
\end{align*}
provided we can find $p,q\in (2,\infty]$ with $2/p+2/q=1$ such that the last expression is finite. Above we have used the fact that $\mu_{\gamma_0}^{\eps,z}$ are distributed like $\nu_{\gamma_0,B_1(0)}$, and denoted by $\nu_{\gamma_0,i}$, $i=1,2$ two independent copies of this measure. By \Cref{lemma: even moments of Y} the first two factors are finite. Therefore, the threshold in \Cref{thm: Existence and moments of star scale invariant GMC} for the existence of the moments of $\nu_\gamma$ determines $p$ and $q$. We must have
\begin{align*}
\begin{cases}
\frac{\gamma}{\gamma_0}(1+\delta)q<\frac{2d}{\gamma_0^2}
\\
\quad\quad\frac{2}{q}+\frac{2}{p}=1
\end{cases}
\Leftrightarrow
\quad
\begin{cases}
\delta<\frac{2d}{\gamma\gamma_0 q}-1
\\
q=\frac{2}{1-\frac{2}{p}}.
\end{cases}
\end{align*}
For $\gamma_0,\gamma\in[0,\sqrt{d})$ these equations can always be satisfied with some $\delta >0$ and $p\in [2,\infty)$.
\end{proof}
\noindent
Since 
\Cref{eq:upper bound for the covariance exponential} together with the upper bounds in the last lemma justify the use of Fubini and dominated convergence in \Cref{eq: gamma_0 difference of squares} we are done.
\end{proof}

\bibliographystyle{plain}

\end{document}